\documentclass[12pt,a4paper]{article}
\usepackage[T2A]{fontenc}
\usepackage[cp1251]{inputenc}
\usepackage[russian, english]{babel}
\usepackage{epsfig,amssymb,epic,eepic,color, graphicx}
\usepackage{fancyhdr}
\usepackage{amsthm}
\usepackage{amsmath}
\usepackage{indentfirst}
\usepackage{setspace}
\usepackage{fancyhdr}
\newtheorem{theorem}{Theorem}
\newtheorem{definition}{Definition}

\newtheorem{lemma}{Lemma}[section]
\newtheorem{proposition}{Proposition}[section]

\begin{document}
\thispagestyle{empty}

\begin{center}{\it \large Vladislav E. Kruglov\footnote{\normalsize HSE Campus in Nizhny Novgorod, Faculty of Informatics, Mathematics, and Computer Science, Laboratory of Topological Methods in Dynamics. Trainee Researcher. Lobachevsky State University of Nizhny Novgorod, Institute ITMM, Department of Mathematical Physics and Optimal Control. Master's student; E-mail: vekruglov@hse.ru}, Dmitry S. Malyshev\footnote{\normalsize HSE Campus in Nizhny Novgorod, Laboratory of Algorithms and Technologies for Networks Analysis. Leading Research Fellow. HSE Campus in Nizhny Novgorod,  Faculty of Informatics, Mathematics, and Computer Science, Department of Applied Mathematics and Informatics. Professor.
Lobachevsky State University of Nizhny Novgorod, Institute ITMM, Department of Algebra, Geometry and Discrete Mathematics. Professor; Doctor of Sciences; E-mail: dmalishev@hse.ru}, Olga V. Pochinka\footnote{\normalsize HSE Campus in Nizhny Novgorod, Faculty of Informatics, Mathematics, and Computer Science, Department of Fundamental Mathematics. Professor, Department Head. HSE Campus in Nizhny Novgorod, Faculty of Informatics, Mathematics, and Computer Science, Laboratory for Topological Methods in Dynamics. Laboratory Head; Doctor of Sciences; E-mail: opochinka@hse.ru}}\end{center}

\vskip 0.35 true cm

\begin{center}\Large{{\bf TOPOLOGICAL CLASSIFICATION OF $\Omega$-STABLE FLOWS ON SURFACES BY MEANS OF EFFECTIVELY DISTINGUISHABLE MULTIGRAPHS\footnote{The authors are grateful to the participants of the seminar ``Topological methods in dynamics'' for fruitful discussions. The classification and realisation results (Sections 1--8 without Subsection 5.3) were obtained with the support of the Basic Research Program at the HSE (project 90) in 2017. The algorithmic results (Subsection 5.3, Section 9) were obtained with the support of Russian Foundation for Basic Research 16-31-60008-mol-a-dk and RF President grant MK-4819.2016.1.}}}\end{center}

\vskip 0.25 true cm

\noindent Structurally stable (rough) flows on surfaces have only finitely many singularities and finitely many closed orbits, all of which are hyperbolic, and they have no trajectories joining saddle points. The violation of the last property leads to $\Omega$-stable flows on surfaces, which are not structurally stable. However, in the present paper we prove that a topological classification of such flows is also reduced to a combinatorial problem. Our complete topological invariant is a multigraph, and we present a polynomial-time algorithm for the distinction of such graphs up to an isomorphism. We also present a graph criterion for orientability
of the ambient manifold and a graph-associated formula for its Euler characteristic. Additionally, we give polynomial-time algorithms for checking the orientability and calculating the characteristic.

\vskip 0.5 true cm


\vskip 0.5 true cm


\section{Introduction}

A traditional method of qualitative studying of a flows dynamics with a finite number of special trajectories on surfaces consists of a splitting the ambient manifold by regions with a predictable trajectories behavior known as {\it cells}. Such a view on continuous dynamical systems rises to the classical work by A. Andronov and L. Pontryagin \cite{AP} published in 1937. In that paper, they considered a system of differential equations 
\begin{equation}\label{odna}
\dot x=v(x),
\end{equation}
where $v(x)$ is a $C^1$-vector field given on a disc bounded by a curve without a contact in the plane and found a roughness criterion for the system (\ref{odna}).

A more general class of flows on the 2-sphere was considered in works by E. Leontovich-Andronova and A. Mayer \cite{LeMa1,LeMa2}, where a topological classification of such flows was also based on splitting by cells, whose types and relative positions (\emph{the Leontovich-Mayer scheme}) completely define a qualitative decomposition of the  phase space of the dynamical system into trajectories. The main difficulty in generalisations of this result to flows on arbitrary orientable surfaces is the possibility of new types of trajectories, namely {\it unclosed recurrent trajectories}. The absence of non-trivial recurrent trajectories for rough flows on the plane and on the sphere is an immediate corollary from the  Poincar\'{e}-Bendixson theory for these surfaces, but this is not so trivial for orientable surfaces of genus $g>0$. At first, it was proved by A.~Mayer ~\cite{Ma} in 1939 for rough flows with no singularities on the 2-torus\footnote{Actually in \cite{Ma} A.~Mayer found the conditions of roughness for cascades (discrete dynamical systems) on the circle and he also got the topological classification for these cascades.} and later by M.~Peixoto~\cite{Pe1, Pe2}  for  structurally stable\footnote{The term ``rough system'' introduced by A. Andronov and L. Pontryagin in \cite{AP} is slightly different from its English counter part ``structurally stable system'' introduced by M. Peixoto in \cite{Pe1, Pe2}, but the sets of rough and structural stable systems coincide.} flows on  surfaces of any genus (see also \cite{PaMe}).

In 1971, M. Peixoto obtained a topological classification of structurally stable flows on arbitrary surfaces \cite{Peix}. As before, he did it by studying all admissible cells and he introduced a combinatorial invariant called {\it a directed graph} generalizing the Leontovich-Mayer scheme. In 1976, D. Neumann and T. O'Brien \cite{NeO}  considered the so-called {\it regular flows} on arbitrary surfaces, such flows have no non-trivial periodic trajectories (i.e. periodic trajectories other than limit cycles) and include the flows above as a particular case. They introduced a complete topological invariant for the regular flows named {\it an orbit complex}, which is a space of flow orbits equipped with some additional information.

In 1998, A. Oshemkov and V. Sharko \cite{OS} introduced a new invariant for Morse-Smale flows on surfaces, namely {\it a three-colour graph}, and described an algorithm to distinct such graphs, which was not, however, polynomial, i.e. its working time is not limited by some polynomial on the length of input information. In the same work they obtained a complete topological classification of Morse-Smale flows on surfaces in terms of atoms and molecules introduced in the work of A. Fomenko \cite{Fom}.

Structurally stable (rough) flows on surfaces have only finitely many singularities and finitely many closed orbits, all of which are hyperbolic, they also have no trajectories joining saddle points. The violation of the last property leads to $\Omega$-stable flows on surfaces, which are not structural stable. However, in the present paper we prove that a topological classification of such flows is also reduced to a combinatorial problem. The complete topological invariant is an equipped graph and we give a polynomial-time algorithm for the distinction of such graphs up to isomorphism. We also present a graph criterion for orientability
of the ambient manifold and a graph-associated formula for its Euler characteristic. Additionally, we give polynomial-time algorithms for checking the orientability and calculating the characteristic.   

\section{The dynamics of an $\Omega$-stable flow}

Let $\phi^t$ be some $\Omega$-stable flow on a closed surface $S$. The non-wandering set $\Omega_{\phi^t}$ of the flow $\phi^t$ consists of a finite number of hyperbolic fixed points and hyperbolic closed trajectories (limit cycles), which are called {\it basic sets}. 

Denote by $G$ a class of $\Omega$-stable flows $\phi^t$ with at least one fixed saddle point or at least one limit cycle\footnote{If flow $\phi^t$ has neither fixed saddle points nor closed trajectories, then its non-wandering set consists of exactly two fixed points: a source and a sink, all such flows are topologically equivalent, that is the reason we exclude such flows from the class $G$.} on a surface $S$. That is the flow class we consider in our work.

\subsection{Fixed points} Let $\phi^t\in G$.

The hyperbolicity of the fixed points is expressed by the following fact. 

\begin{proposition}[\cite{PaMe}, Theorem 5.1 from Chapter 2 and \cite{Clark}, Theorem 7.1 from Chapter 4]\label{LokSopr}
The flow $\phi^t$ in some neighbourhood of a fixed point $q\in\Omega_{\phi^t}$ is topologically equivalent to one of the following linear flows 
\begin{align*}
a^t(x,y)=&\left(2^{-t}x,2^{-t}y\right),\\
b^t(x,y)=&\left(2^{-t}x,2^ty\right),\\ 
c^t(x,y)=&\left(2^tx,2^ty\right).
\end{align*}
\end{proposition} 

In the cases $a^t$, $b^t$, $c^t$ the fixed point $q$ is called {\it sink, saddle, source} and has the dimension of the unstable manifold $W^u_q$ equal to $0,1,2$ accordingly. We will denote by $\Omega^0_{\phi^t}$, $\Omega^1_{\phi^t}$, $\Omega^2_{\phi^t}$ the set of all sinks, saddles, sources of $\phi^t$ accordingly.    

It follows from the criterion of the $\Omega$-stability in \cite{Pugh} that the saddle points do not organize {\it cycles}, i.e. collections of points 
\begin{equation*} 
q_1,\dots,q_k,q_{k+1}=q_1 
\end{equation*} 
with a property 
\begin{equation*}
W^s_{q_i}\cap W^u_{q_{i+1}}\neq\emptyset,\;i=1,\dots,k.
\end{equation*}

\subsection{Closed trajectories}

Let $\mathfrak c$ be a closed trajectory of $\phi^t$ and $p\in\mathfrak c$. Let $\Sigma_p$ be a smooth cross-section passing through the point $p$ transversal to trajectories of $\phi^t$ near $p$. Let $V_p\subset\Sigma_p$ be a neighbourhood of $p$ such that for every point $x\in V_p$ the value $\tau_x\in\mathbb R^+$ with properties  $\phi^{\tau_x}(x)\in V_p$ and $\phi^t(x)\notin V_p$ for any $0<t<\tau_x$ is well-defined. Then $\Sigma_p$ is called a {\it Poincar\'e's cross-section} and a map $F_p\colon V_p\to\Sigma_p$ given by the formula $F_p(x)=\phi^{\tau_x}(x),\,x\in V_p$ is called {\it Poincar\'e's map}.

The hyperbolicity of the closed trajectory $\mathfrak c$ is expressed by the following fact.

\begin{proposition}[\cite{PaMe}, Proposition 1.2 from Chapter 3 and Theorem 5.5 from Chapter 2]\label{LokSoprDif} Poincar\'e's map $F_p\colon V_p\to F_p(V_p)$ is a diffeomorphism with a fixed point $p$ in a neighbourhood of which $F_p$ is topologically conjugate to one of the following linear diffeomorphisms 
\begin{align*}
a_{+}(x)=\frac{x}{2},\,&\,a_{-}(x)=-\frac{x}{2},\\
c_{+}(x)=2x,\,&\,c_{-}(x)=-2x.
\end{align*}
\end{proposition}

In the cases $a_\pm,\,c_\pm$ the closed trajectory $\mathfrak c$ is called {\it attractive, repelling limit cycle} accordingly. Denote by $\Omega^3_{\phi^t}$ the set of all limit cycles of $\phi^t$.

In any case the limit cycle $\mathfrak c$ has a neighbourhood $U_\mathfrak c$, avoiding other limit cycles and fixed points of $\phi^t$ and with the transversal to the trajectories of $\phi^t$ boundary $R_\mathfrak c$.  The neighbourhood $U_\mathfrak c$ is homeomorphic to the annulus or the M\"obius band (see Fig. \ref{ListMobiusa}) in the cases $a_+,\,c_+$ or $a_-,\,c_-$ accordingly and can be constructed the following way.

\begin{figure}[htb]
\centerline{\includegraphics [width=7 cm]{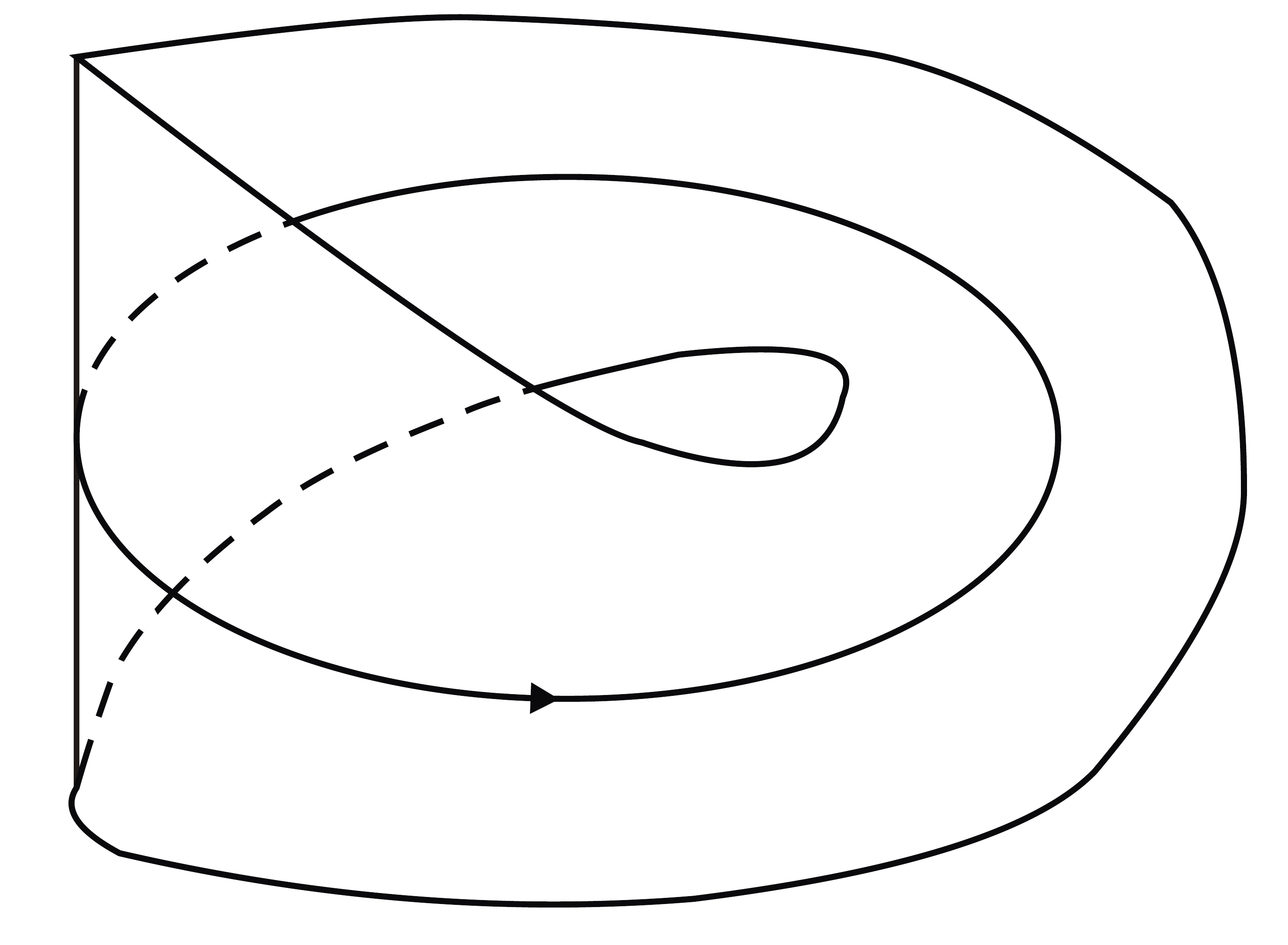}} \caption{The case in which $U_\mathfrak c$ is homeomorphic to the M\"obius band}\label{ListMobiusa}
\end{figure}

For every points $a,b\in V_p$ let us denote by $m_{a,b}$ the segment of $V_p$ bounded by the points $a,b$ and by  $\mu_{a,b}$ the length of this segment. In the cases $a_+,\,c_+$ let us choose  points $x^*_1,x^*_2\in(V_p\setminus \{p\})$ on different connected components of $V_p\setminus \{p\}$. Then $R_\mathfrak{c}$ is a union of two circles 
\begin{equation*}
\left\{\phi^{\frac{\mu_{x,F_p(x^*_1)}}{\mu_{x_1^*,F_p(x_1^*)}}\tau_{x}}(x): x\in m_{x_1^*,F_p(x_1^*)}\right\}\,\,\textrm{ and }\,\,\left\{\phi^{\frac{\mu_{x,F_p(x^*_2)}}{\mu_{x_2^*,F_p(x_2^*)}}\tau_{x}}(x): x\in m_{x_2^*,F_p(x_2^*)}\right\}.
\end{equation*} 
In the cases $a_-,\,c_-$ let us choose a point $x^*\in(V_p\setminus \{p\})$. Then 
\begin{equation*}
R_\mathfrak c=\left\{\phi^{\frac{\mu_{x,F^2_p(x^*)}}{\mu_{x^*,F^2_p(x^*)}}2\tau_{x}}(x): x\in m_{x^*,F^2_p(x^*)}\right\}.
\end{equation*}

A moving of $\Sigma_p$ along the trajectories in the positive time gives a {\it consistent with $\mathfrak c$ orientation on $R_\mathfrak c$}. Thus, in further we will assume that $R_\mathfrak c$ is oriented consistently with $\mathfrak c$. 

\section{The directed graph for a flow $\phi^t\in G$}

Recall that {\it a graph} $\Gamma$ is an ordered pair $(B,E)$ such that $B$ is a finite non-empty set of {\it vertices}, $E$ is a set of pairs of the vertices called {\it edges}. Besides, if $E$ is a multiset then $\Gamma$ is called {\it multigraph}. Recall that a {\it multiset} is a set with the opportunity of multiple inclusion of its elements. Everywhere below we will call a multigraph simply as a {\it graph}. 

If a graph includes an edge $e=(a,b)$, then both vertices $a$ and $b$ are called {\it incident} to the edge $e$. The vertices $a$ and $b$ \emph{are connected by} $e$. A graph is called {\it directed} if every its edge is an ordered pair of vertices. A finite sequence 
\begin{equation*}\tau=(b_0, (b_0,b_1), b_1,\dots, b_{i-1}, (b_{i-1},b_i), b_i,\dots, b_{k-1}, (b_{k-1}, b_k), b_k)
\end{equation*}
of vertices and edges of a graph is called {\it a path}, the number $k$ is called {\it the length} of the path and it is equal to the number of edges of the path. The path $\tau$ is called {\it simple} if it contains only pairwise disjoint edges.  The simple path $\tau$ is called {\it a cycle} if $b_0=b_k$. A graph is called {\it connected} if every two its vertices can be connected by a path.

Let $\mathcal R=\bigcup\limits_{\mathfrak c\in\Omega^3_{\phi^t}}R_\mathfrak c$. We call $\mathcal R$ a {\it cutting set} and the connected components of $\mathcal R$ {\it cutting circles}. Let $\hat{S}=S\backslash\mathcal R$. We call {\it an elementary region} a connected component of the set $\hat{S}$. The elementary regions, obviously, can be of the following pairwise disjoint types with respect to information about basic sets of $\phi^t$ in the regions:

1) a region of the type $\mathcal L$ contains exactly one limit cycle;

2) a region of the type $\mathcal A$  contains exactly one source or exactly one sink;

3) a region of the type $\mathcal M$ contains at least one saddle point;

4) a region of the type $\mathcal E$ does not contain elements of basic sets.

\begin{definition} A directed  graph $\Upsilon_{\phi^t}$ is said to be a \emph{graph of the flow} $\phi^t\in G$ (see Fig. \ref{OmegaUstIEgoNeosnIpsilon}) if

(1) the vertices of $\Upsilon_{\phi^t}$ bijectively correspond to the elementary regions of $\phi^t$;

(2) every directed edge of $\Upsilon_{\phi^t}$, which joins a vertex $a$ with a vertex $b$, corresponds to the cutting circle $R$, which is a common boundary of the regions $A$ and $B$ corresponding to $a$ and $b$, such that any trajectory of $\phi^t$ passing $R$ goes from $A$ to $B$ by increasing the time.
\end{definition}

\begin{figure}[htb]
\centerline{\includegraphics [width=12 cm]{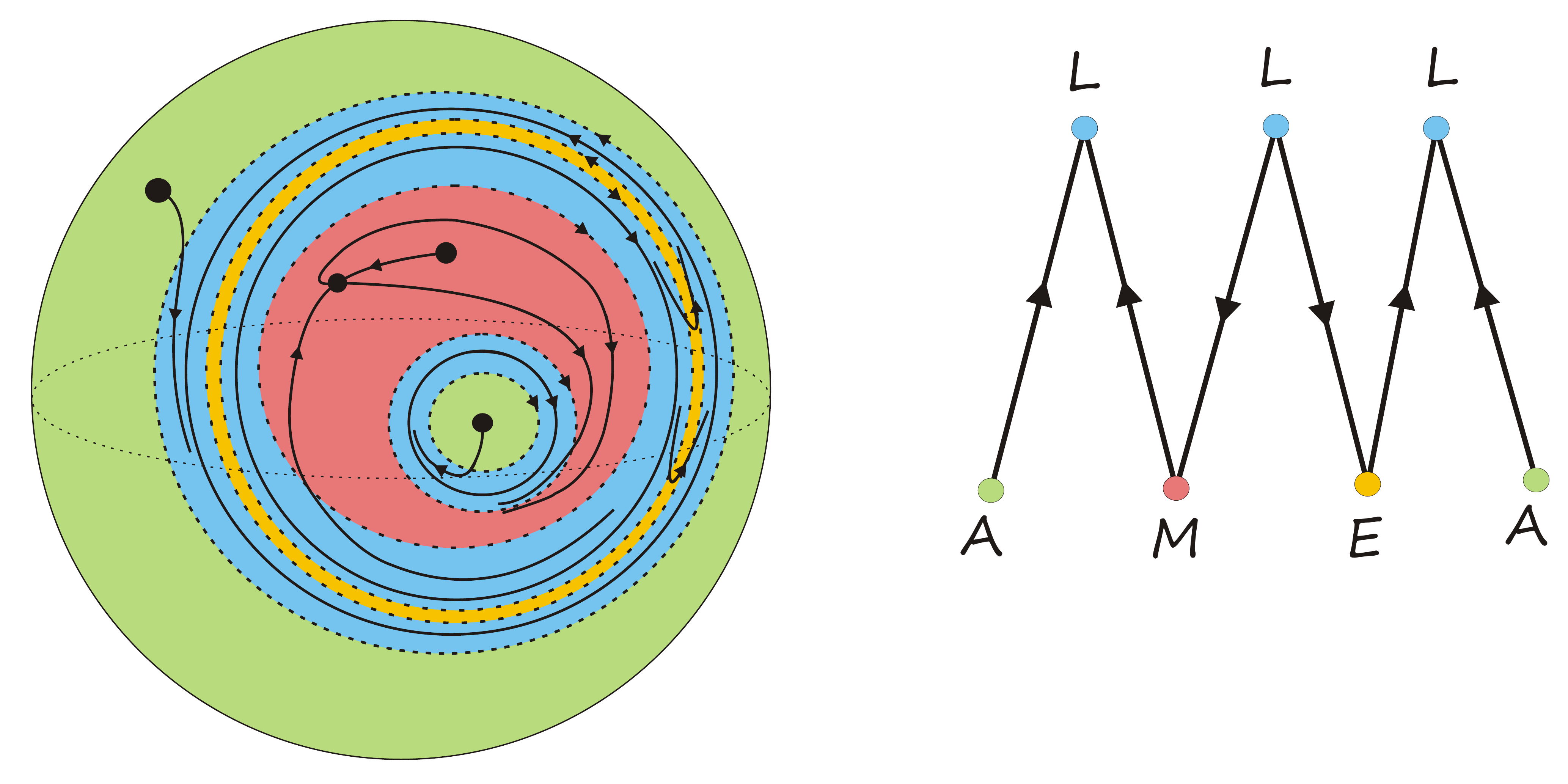}} \caption{$\phi^t$ and $\Upsilon_{\phi^t}$}\label{OmegaUstIEgoNeosnIpsilon}
\end{figure}

We will call a $\mathcal L$-, $\mathcal A$-, $\mathcal E$- or $\mathcal M$\emph{-vertex} a vertex of $\Upsilon_{\phi^t}$, which corresponds to a $\mathcal L$-, $\mathcal A$-, $\mathcal E$- or $\mathcal M$-region accordingly.

The following proposition immediately follows from the dynamics of the flow $\phi^t$ and a structure of cutting set.

\begin{proposition}\label{SvOrGr} Let $\Upsilon_{\phi^t}$ be the directed graph of a flow $\phi^t\in G$, then: 

1) every $\mathcal M$-vertex can be connected only with $\mathcal L$-vertices, furthermore, with every vertex by a single edge;

2) every $\mathcal E$-vertex can be incident only to two edges that connect this vertex with two different $\mathcal L$-vertices, and one of these edges enters to the $\mathcal E$-vertex, another one exits;

3) every $\mathcal A$-vertex can be connected only with a $\mathcal L$-vertex, furthermore, by a single edge;

4) every $\mathcal L$-vertex has degree (the number of incident edges) 1 or 2, and if its degree is 2, then both edges either enter the vertex or exit.
\end{proposition}

The existence of an isomorphism of the directed graphs for  topologically equivalent $\Omega$-stable flows from $G$ is a necessary condition. To make the directed graph a complete topological invariant for the class $G$, below we equip the graph $\Upsilon_{\phi^t}$ by additional information.

\section{Equipment of the directed graph}
In this section, we describe how to assign some additional information to vertices and edges of the directed graph of a flow from $G$.

\subsection{$\mathcal A$-vertex} The flows in $\mathcal A$-regions can belong to only the two equivalence classes: a source pool and a sink pool, which we can distinguish by directions of edges incident to $\mathcal A$-vertices.

\subsection{$\mathcal L$-vertex} The flows in $\mathcal L$-regions can belong to only the four equivalence classes: an annulus with a stable limit cycle,  an annulus with an unstable one, the M\"obius band with a stable one, the M\"obius band with an unstable one, which we can distinguish by directions of edges and by quantities of edges incident to $\mathcal L$-vertices.

\subsection{$\mathcal E$-vertex} The flows in $\mathcal E$-regions can belong to only the two equivalence classes corresponding to the consistent and the inconsistent orientation of connecting components of $\mathcal E$'s boundary. However, a structure of an $\mathcal E$-region cannot be determined by the directed graph, therefore, we will attribute the {\it weight} to the vertex corresponding to an $\mathcal E$-region. The weight is ``$+$'' in the consistent case and ``$-$'' in the inconsistent one.

\begin{figure}[htb]
\centerline{\includegraphics [width=11 cm] {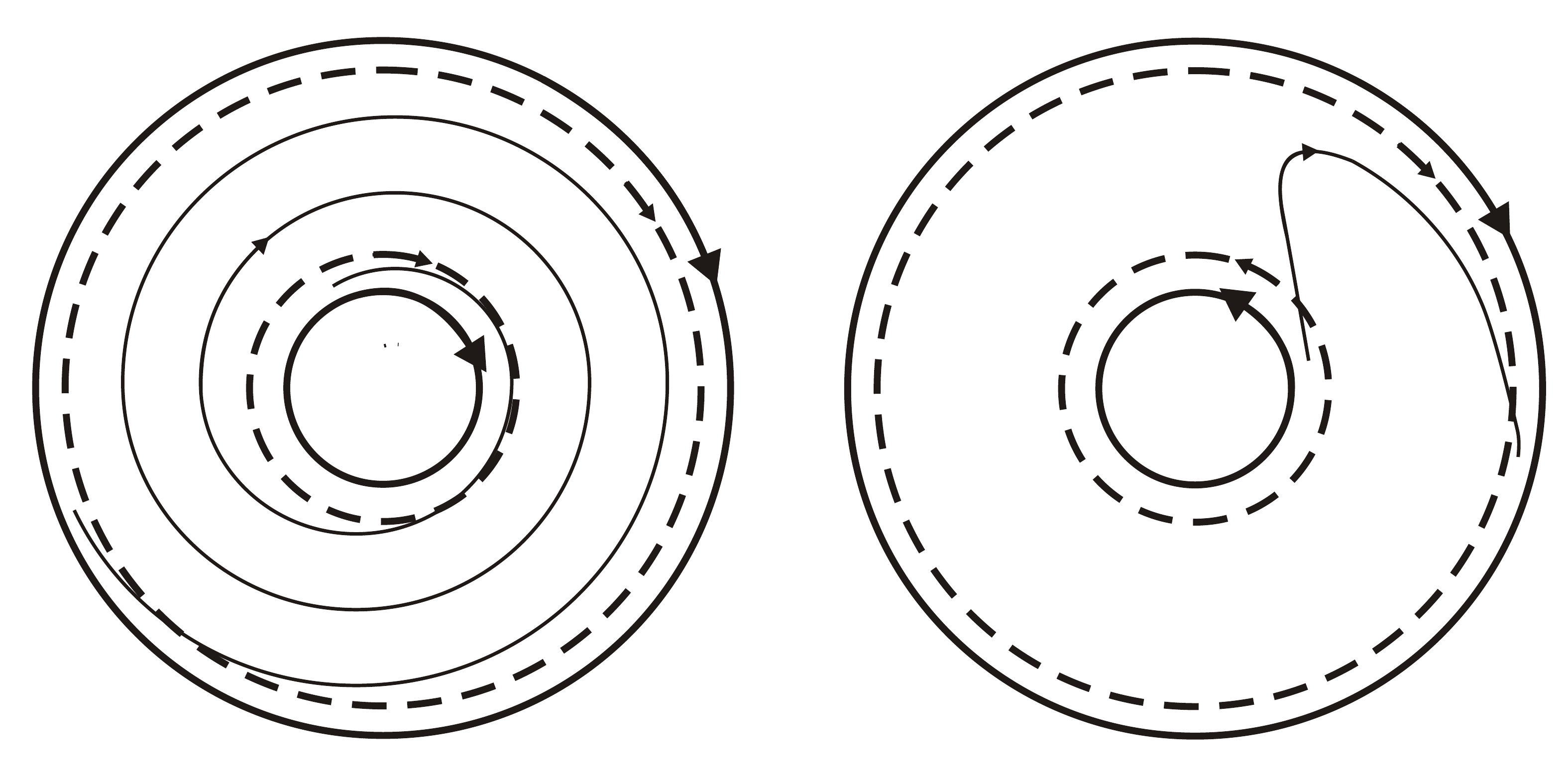}}
\caption{The cases of the consistent (leftward) and the inconsistent (rightward) orientation of boundary's connecting component of some $\mathcal E$-region.}\label{EDvaSl}
\end{figure}

\subsection{$\mathcal M$-vertex}\label{Mvertex} The flows in $\mathcal M$-regions cannot be determined by the directed graph. Then we will equip vertices corresponding to them by {\it four-colour graphs} for a description of the dynamics of the flow in the regions. In more details.

All results about flows from $G$ without periodic trajectories are given and proved in our paper \cite{KruMaPoMS} but we give it here for completness.

Let us consider some $\mathcal M$-region that is either a 2-manifold with a boundary or a closed surface. In the first case let us attach the union $D$ of disjoint 2-disks to the boundary to get a closed surface $M$, in the second case we also denote the closed surface by $M$ and will suppose that $D=\emptyset$. Let us extend  $\phi^t|_{\mathcal M}$ up to an $\Omega$-stable flow $f^t\colon M\to M$ assuming that $f^t$ coincides with $\phi^t$ out of $D$ and $\Omega_{f^t}$ has exactly one fixed point (a sink or a source) in each connected component of $D$.

Let $\Omega^0_{f^t},\,\Omega^1_{f^t},\,\Omega^2_{f^t}$ be the sets of all sources, saddle points and sinks of $f^t$ accordingly. By the definition of the region $\mathcal M$ the flow $f^t$ has at least one saddle point. Let 
\begin{equation*}
\tilde M=M\setminus (\Omega^0_{f^t}\cup W^s_{\Omega^1_{f^t}}\cup W^u_{\Omega^1_{f^t}}\cup \Omega^2_{f^t}).
\end{equation*} 
A connected component of $\tilde M$ is called {\it a cell}. 

\begin{lemma}\label{1ist1st}
Every cell $J$ of the flow $f^t$ contains a single sink $\omega$ and a single source $\alpha$ in its boundary, and the whole cell is the union of trajectories going from $\alpha$ to $\omega$.
\end{lemma}

Let us choose a trajectory $\theta_J$ in the cell $J$, we will call it a {\it $t$-curve}. Let 
\begin{equation*}
\mathcal T=\bigcup\limits_{J\subset\tilde S}\theta_J,\,\,\bar M=\tilde{M}\backslash\mathcal T.
\end{equation*}

\begin{figure}[htb]
\centerline{\includegraphics [width=7.5 cm] {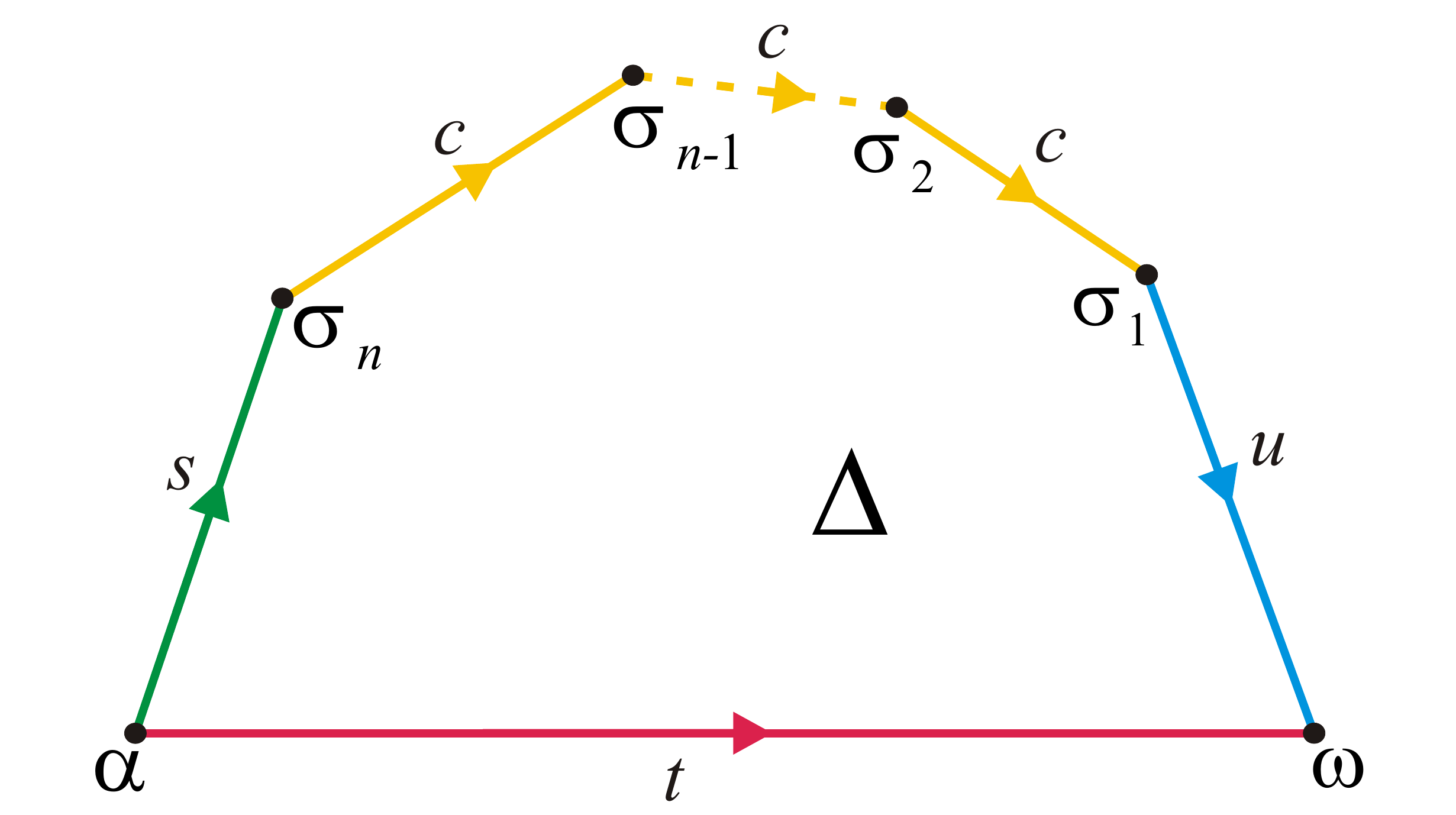}}
\caption{A polygonal region}\label{MnogObl}
\end{figure}

Let us call a {\it $c$-curve} a separatrix connecting saddle points (from the word ``connection''), a {\it $u$-curve} an unstable saddle separatrix with a sink in its closure, a {\it $s$-curve} a stable saddle separatrix with a source in its closure. We will call a {\it polygonal region $\Delta$} the connecting component of $\bar M$.

\begin{lemma}\label{vidmnob}
Every polygonal region $\Delta$ is homeomorphic to an open disk and its boundary consists of an unique $t$-curve, an unique $u$-curve, an unique $s$-curve, and a finite (may be empty) set of $c$-curves (see Fig. \ref{MnogObl}).
\end{lemma}

Denote by $\Delta_{f^t}$ the set of all polygonal regions of $f^t$ (see Fig. \ref{SlozhnPrIMnObl}, where a flow $f^t$ and all its polygonal regions are presented). 

\begin{figure}[htb]
\centerline{\includegraphics [width=15 cm] {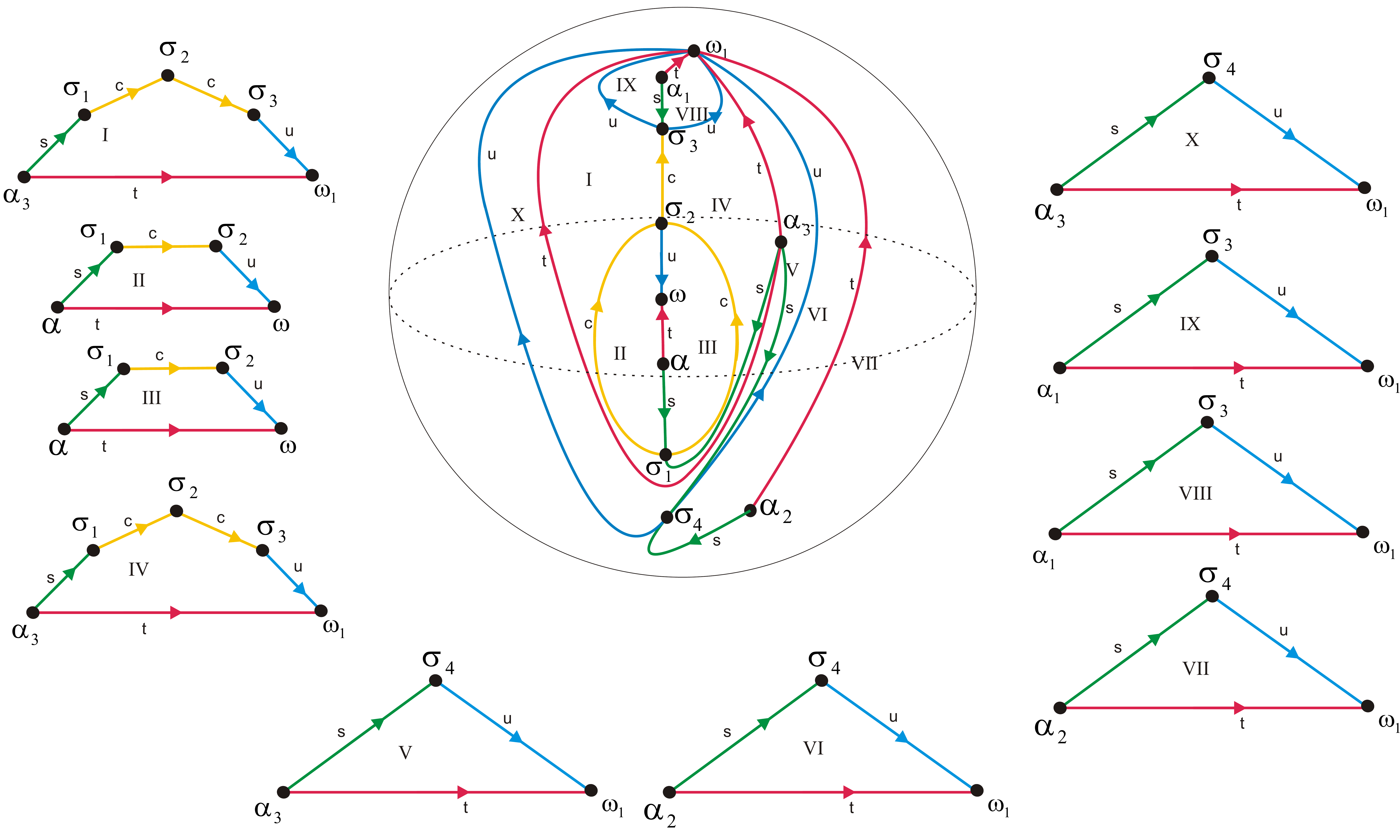}} \caption{An example of flow $f^t$ together with the polygonal regions}\label{SlozhnPrIMnObl}
\end{figure}

\begin{definition}
A multigraph is called {\it $n$-colour graph} if the set of its edges is the disjoint union of $n$ subsets, each of which consists of edges of the same colour.
\end{definition}

We say that a four-colour graph $\Gamma_{\mathcal M}$ with edges of colours $u,~s,~u,~t$ bijectively corresponds to $f^t$ if:

1) the vertices of $\Gamma_{\mathcal M}$ bijectively correspond to the polygonal regions of $\Delta_{f^t}$;

2) two vertices of $\Gamma_{\mathcal M}$ are incident to an edge of colour $s$, $t$, $u$ or $c$ if the polygonal regions corresponding to these vertices has a common $s$-, $t$-, $u$- or $c$-curve; that establishes an one-to-one correspondence between the edges of $\Gamma_{\mathcal M}$ and the colour curves;

3) if some vertex  $b$ of $\Gamma_{\mathcal M}$ is incident to more than one $c$-edge (the number $n_b$ of $c$-edges is more than $1$), then $c$-edges are ordered 
\begin{equation*}
c^b_1,\dots,c_{n_b}^b
\end{equation*} 
by a moving (according to the direction from the source to the sink on $t$-curve) along the boundary of the corresponding polygonal region (see, for example, Figure \ref{GrafS4Tipom}). 

\begin{definition} We say that the graph $\Gamma_{\mathcal M}$ is the four-colour graph of the flow $f^t$ corresponding to $\phi^t|_{\mathcal M}$.
\end{definition}

\begin{figure}[htb]
\centerline{\includegraphics [width=15 cm] {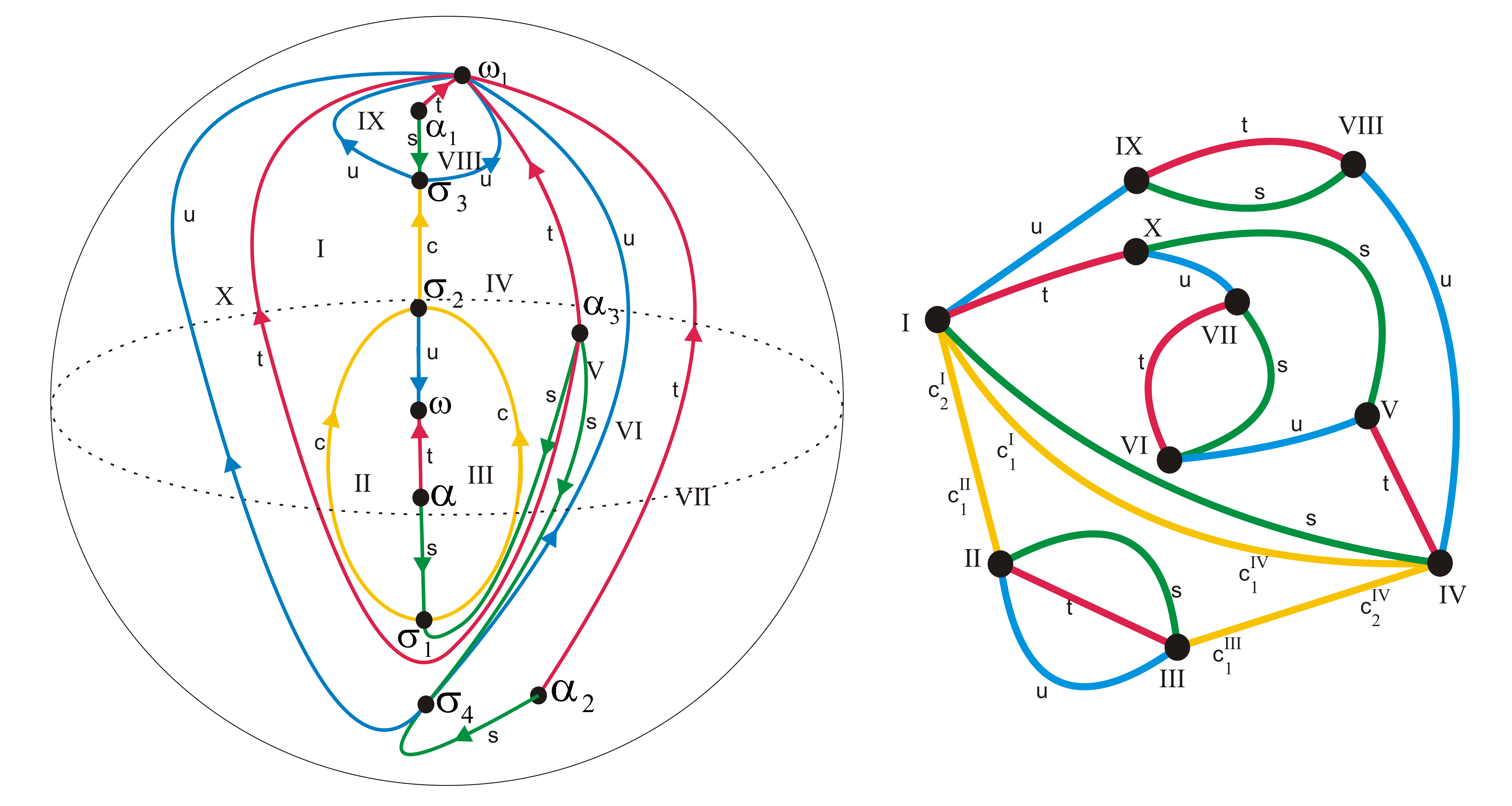}} \caption{An example of $f^t$ and its four-colour graph}\label{GrafS4Tipom}
\end{figure}

\begin{definition} Two four-colour graphs $\Gamma_{\mathcal M}$ and $\Gamma_{\mathcal M'}$ corresponding to $\phi^t|_{\mathcal M}$ and $\phi'^t|_{\mathcal M'}$ respectively are said to be {\it isomorphic} if there is an one-to-one correspondence $\psi$ of vertices and edges of the first graph to vertices and edges of the second graph preserving colours of all edges and numbers of $c$-edges.
\end{definition}

\subsection{$(\mathcal M,\mathcal L)$- and $(\mathcal L,\mathcal M)$-edge}\label{MLedge}

Let us denote by $\pi_{f^t}$ the one-to-one correspondence described above between polygonal regions and vertices, also between colour curves of $f^t$ and colour edges of $\Gamma_{\mathcal M}$ respectively.

Let us call a {\it $st$-cycle} ({\it $tu$-cycle}) a cycle of $\Gamma_{\mathcal M}$ consisting only of $s$- and $t$-edges ($t$- and $u$-edges). Let us call $u$- and $s$-edges exiting out a vertex $b$ as \emph{nominal $c$-edges} and assign the numbers $0$ and $n_b+1$ to them respectively. Let us call a {\it $c^*$-cycle} a simple cycle 
\begin{equation*}b_1, (b_1, b_2), b_2,\dots, b_{2k}, b_{2k+1}, b_{2k+1}=b_1,
\end{equation*}
 if 
\begin{equation*}(b_{2i-1}, b_{2i})=c^{b_{2i}}_m,\,(b_{2i}, b_{2i+1})=c^{b_{2i}}_{m+1}=c^{b_{2i+1}}_{l},\,(b_{2i+1}, b_{2i+2})=c^{b_{2i+1}}_{l-1}.
\end{equation*}

\begin{proposition}\label{TochkiICikly} The projection $\pi_{f^t}$ gives an one-to-one correspondence between the sets $\Omega^0_{f^t}$, $\Omega^1_{f^t}$, $\Omega^2_{f^t}$ and the sets of $tu$-, $c^*$-, and $st$-cycles respectively.
\end{proposition}

By our construction $M=\mathcal M\cup D$, where $D$ is either empty or each its connected component contains exactly one sink $\omega$ (source $\alpha$) of the flow $f^t$, uniquely corresponding to a cutting circle $R_\mathfrak{c}$ for a limit cycle $\mathfrak{c}$ of the flow $\phi^t$, which uniquely corresponds to a $(\mathcal M,\mathcal L)$-edge ($(\mathcal L,\mathcal M)$-edge) of the graph $\Upsilon_{\phi^t}$. Due to Proposition \ref{TochkiICikly} the node $\omega$ ($\alpha$) uniquely corresponds to a $tu$-cycle (a $st$-cycle), denote it by $\tau_{_{\mathcal M,\mathcal L}}$ ($\tau_{_{\mathcal L,\mathcal M}}$). Moreover, due to Proposition \ref{TochkiICikly}, we can embed the graph $\Gamma_{\mathcal M}$ such that the cycle $\tau_{_{\mathcal M,\mathcal L}}$ ($\tau_{_{\mathcal L,\mathcal M}}$) coincides with $R_\mathfrak{c}$. Thus we induce an orientation from $R_\mathfrak{c}$ to the cycle and call the cycle $\tau_{_{\mathcal M,\mathcal L}}$ ($\tau_{_{\mathcal L,\mathcal M}}$) {\it oriented} one. 

\section{The formulation of the results}

\begin{definition}\label{Y*} Let $\Upsilon_{\phi^t}$ be the directed graph of a flow $\phi^t\in G$. We will say that $\Upsilon_{\phi^t}$ is the {\it equipped graph} of $\phi^t$ and denote it by $\Upsilon^*_{\phi^t}$ if:

(1) every $\mathcal E$-vertex is equipped with the weight ``$+$'' or ``$-$'' in consistent and inconsistent case respectively;

(2) every $\mathcal M$-vertex is equipped with a four-colour graph $\Gamma_{\mathcal M}$ corresponding to the flow $f^t$ constructed in Subsection \ref{Mvertex};

(3) every edge $(\mathcal M,\mathcal L)$ ($(\mathcal L,\mathcal M)$) is equipped with an oriented $tu$-cycle  ($st$-cycle) $\tau_{_{\mathcal M,\mathcal L}}$ ($\tau_{_{\mathcal L,\mathcal M}}$) of $\Gamma_{\mathcal M}$ corresponding to the limit cycle $\mathfrak c$ of $\mathcal L$ and oriented consistently with $R_\mathfrak c$ (see Fig. \ref{OmegaUstIEgoYpsilonDvaPrimera}).
\end{definition}

\begin{figure}[htb]
\centerline{\includegraphics [width=15 cm]{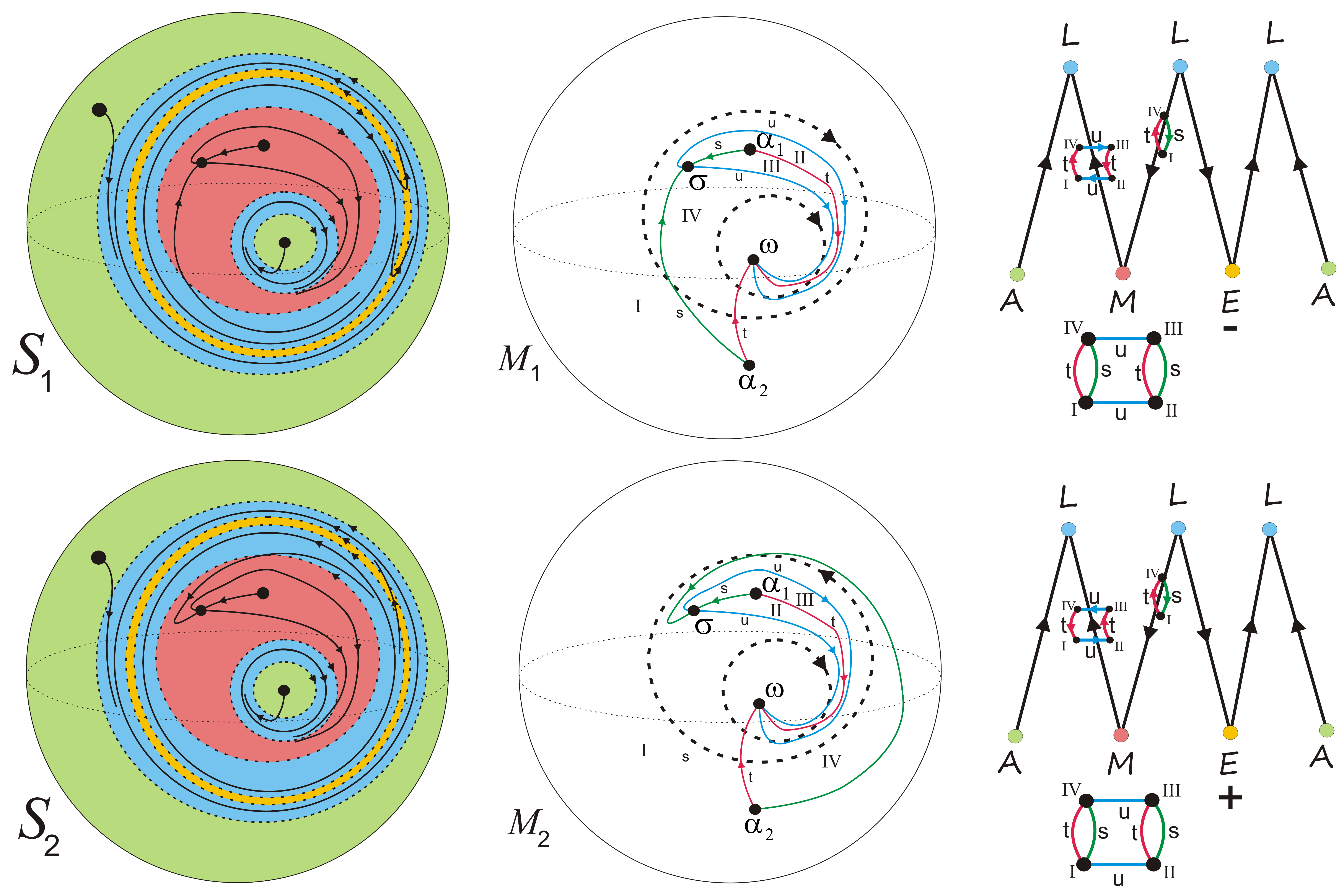}} \caption{Two flows from $G$ and their equipped graphs}\label{OmegaUstIEgoYpsilonDvaPrimera}
\end{figure}

On Fig. \ref{DvaPrimeraRaznOrientMezhduMiA} you can see the two examples of flows from $G$ whose difference might be defined only by oriented cycles equipping their graphs, and on Fig. \ref{DvaPrimeraNaTore} -- by weight of $\mathcal E$-vertices.

\begin{figure}[htb]
\centerline{\includegraphics [width=15 cm]{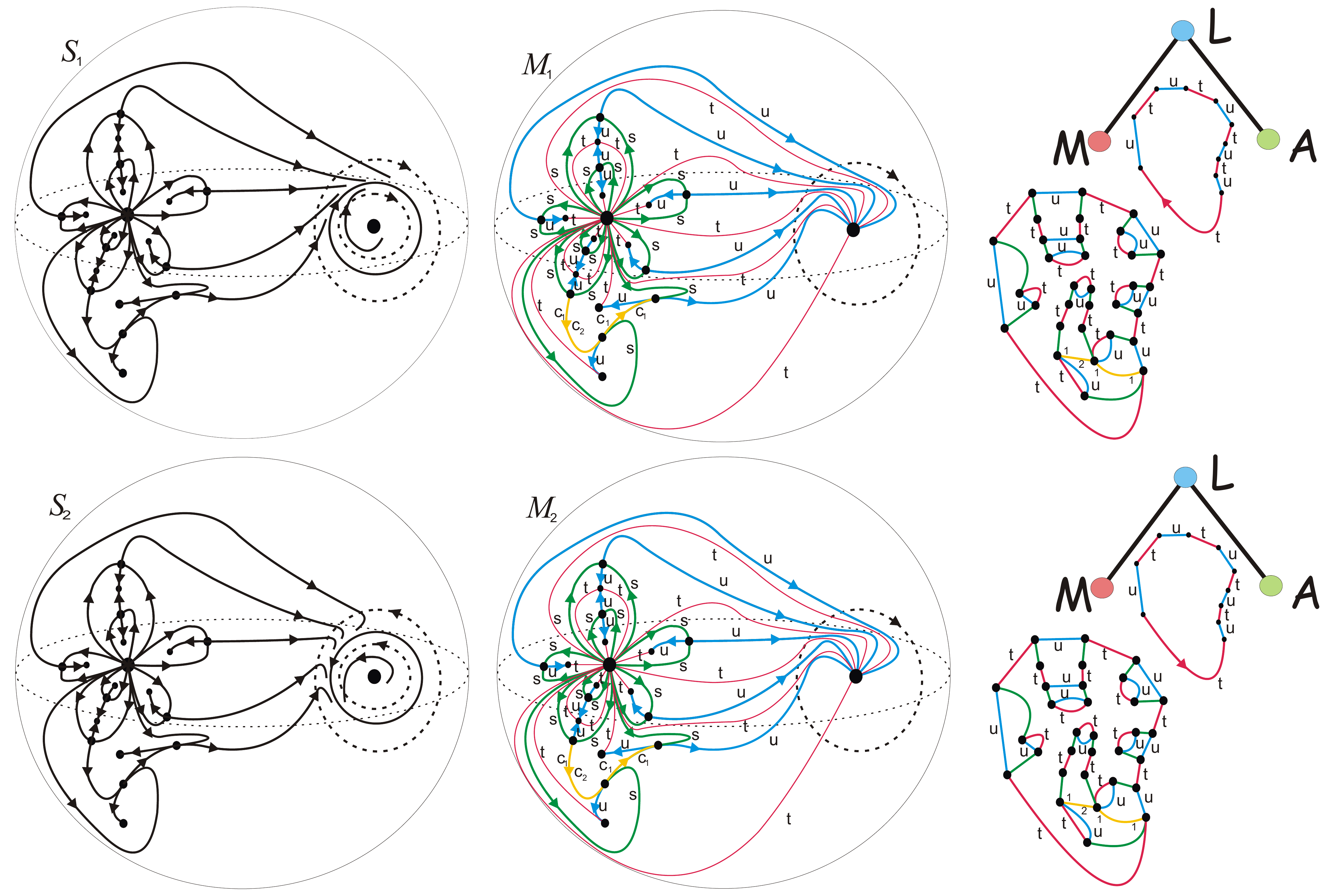}} \caption{Two examples of flows from $G$ differing only by orientation of the limit cycle between $\mathcal M$ and $\mathcal A$ and their equipped graphs}\label{DvaPrimeraRaznOrientMezhduMiA}
\end{figure}

\begin{figure}[htb]
\centerline{\includegraphics [width=9 cm]{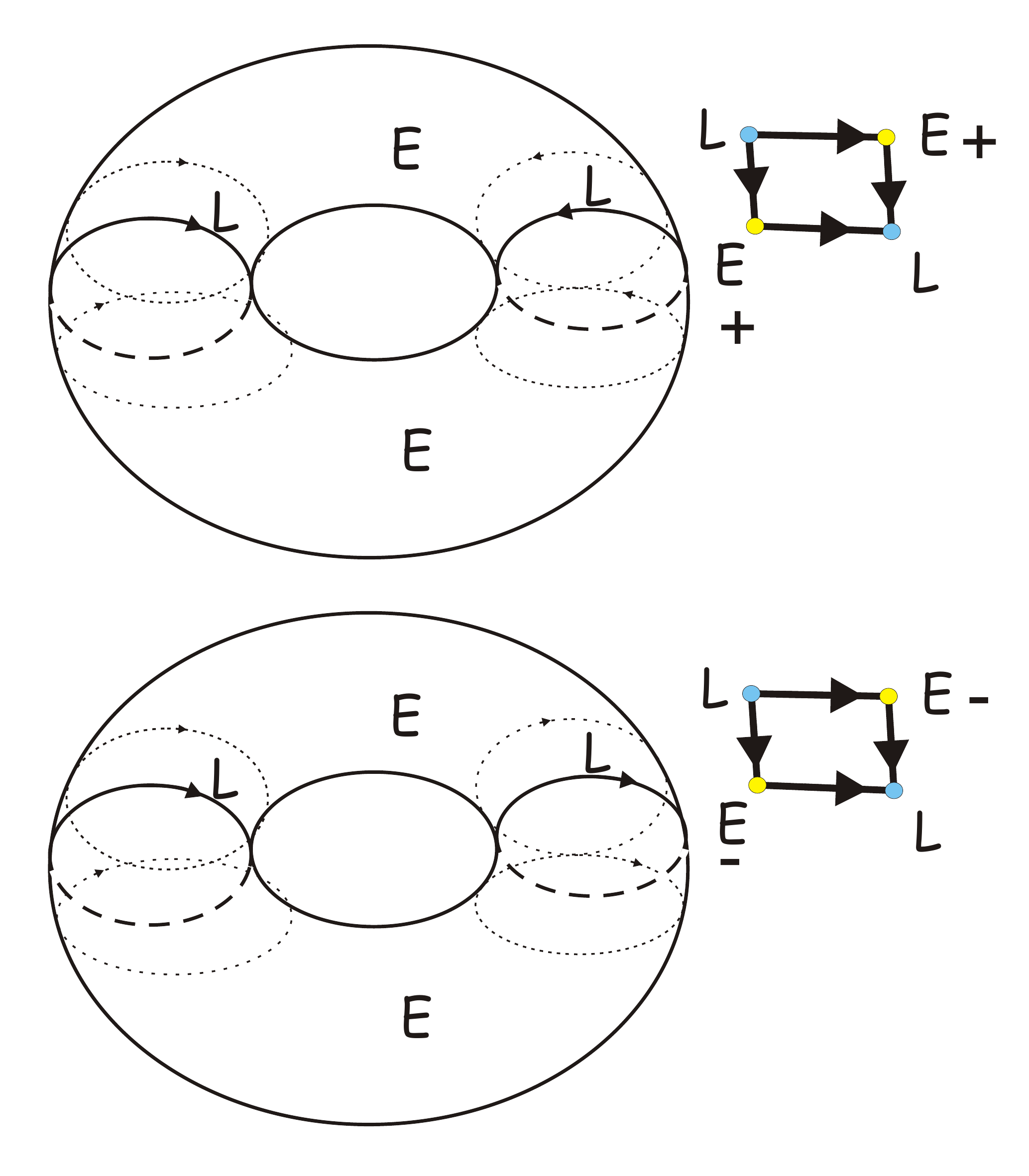}} \caption{Two examples of flow from $G$ without $\mathcal A$- and $\mathcal M$-regions differing only by orientation of the limit cycle and their equipped graphs}\label{DvaPrimeraNaTore}
\end{figure}

Let us denote by $\pi^*_{\phi^t}$ the one-to-one correspondence described above between the elementary regions and the vertices, the cutting circles and the edges, the directions of the trajectories and the directions of the edges, the consistencies of the orientations of the boundary's connecting components of $\mathcal E$-regions and the weights of the $\mathcal E$-vertices, the $\mathcal M$-regions and the four-colour graphs, the stable limit cycles and the $tu$-cycles, the unstable limit cycles and the $st$-cycles, the orientations of the stable limit cycles and the orientations of the cycles $\tau_{_{\mathcal M,\mathcal L}}$, the orientations of the unstable limit cycles and the orientations of the cycles $\tau_{_{\mathcal L,\mathcal M}}$ accordingly.

\subsection{The classification result}

\begin{definition} Equipped graphs $\Upsilon^*_{\phi^t}$ and $\Upsilon^*_{\phi'^t}$ are said to be {\it isomorphic} if there is an one-to one correspondence $\xi$ between all edges and vertices of $\Upsilon^*_{\phi^t}$ and all edges and vertices of $\Upsilon^*_{\phi'^t}$ preserving their equipments in the following way:

(1) the weights of vertices $\mathcal E$ and $\xi(\mathcal E)$ are equal;

(2) for vertices $\mathcal M$ and $\xi(\mathcal M)$, there is an isomorphism $\psi_{\mathcal M}$ of the four-colour graphs $\Gamma_{\mathcal M},\,\Gamma_{\xi(\mathcal M)}$ such that $\psi_{\mathcal M}(\tau_{_{\mathcal M, \mathcal L}})=\tau_{_{\xi(\mathcal M),\xi(\mathcal L)}}$ and the orientations of $\psi_{\mathcal M}(\tau_{_{\mathcal M, \mathcal L}})$ and $\tau_{_{\xi(\mathcal M),\xi(\mathcal L)}}$ coincide (similarly for $\tau_{_{\mathcal L, \mathcal M}}$).
\end{definition}

\begin{theorem}\label{Classif} Flows $\phi^t, \phi'^t\in G$ are topologically equivalent if and only if the equipped graphs $\Upsilon^*_{\phi^t}$ and $\Upsilon^*_{\phi'^t}$ are isomorphic.
\end{theorem}

\subsection{The realisation results}

To solve the realization problem, we introduce the notion of an admissible four-colour graph and an equipped graph.

Let $\Gamma$ be a four-colour graph with the properties:

(1) every edge of the four-colour graph is coloured in one of the four colors: $s, u, t, c$; 

(2) every vertex of the four-colour graph is incident to exactly one edge of the colours $s, u, t$. Besides, the number $n_b$ of $c$-edges incident to a vertex $b$ can be any (may be null) and these edges $c^b_1,\dots,c^b_{n_b}$ are ordered if $n_b\geq 1$.

\begin{definition}\label{afcg} The graph $\Gamma$ we call an {\it admissible four-colour graph} if it contains $c^*$-cycles and every such a cycle has four vertices.
\end{definition}

\begin{lemma}\label{GmAdm} The graph $\Gamma_{\mathcal M}$ is admissible.
\end{lemma}

\begin{lemma}\label{ReBezPer} Every admissible four-colour graph $\Gamma$ corresponds to a closed surface $M$ and an $\Omega$-stable flow $f^t\colon M\to M$ from $G$ without limit cycles, besides:

(1) The Euler characteristic of $M$ can be calculated by the formula 
\begin{equation}\label{EjHarBezPer}
\chi(M)=\nu_0-\nu_1+\nu_{2},
\end{equation} 
where $\nu_0, \nu_1, \nu_{2}$ are the numbers of all $tu$-, $c^*$- and $st$-cycles of $\Gamma$ respectively;

(2) $M$ is non-orientable if and only if $\Gamma$ has at least one cycle with an odd length.
\end{lemma}

\begin{definition} We call $\Upsilon^*$ an {\it admissible equipped graph} if it is a connected directed graph $\Upsilon$ with $\mathcal A$-, $\mathcal L$-, $\mathcal E$- and $\mathcal M$-vertices satisfying the items (1)--(4) of Proposition \ref{SvOrGr}, so that 

-- every $\mathcal M$-vertex is equipped with an admissible four-colour graph $\Gamma_{\mathcal M}$, 

-- every edge entering to (exiting out of) any $\mathcal M$-vertex is equipped with an oriented $st$-cycle ($ut$-cycle) of the four-colour graph, 

-- every $\mathcal E$-vertex is assigned with a weight ``$+$'' or ``$-$''.
\end{definition}

\begin{lemma}\label{Dopust} The graph $\Upsilon^*_{\phi^t}$ is admissible.
\end{lemma}

For every $\mathcal M$-vertex of an admissible equipped graph $\Upsilon^*$, let us denote by $X_{\mathcal M}$ the result of applying the formula (\ref{EjHarBezPer}) to the corresponding admissible four-colour graph $\Gamma_{\mathcal M}$. Denote by $Y_{\mathcal M}$ the quantity of edges, which are incident to $\mathcal M$ and denote by $N_\mathcal A$ the quantity of $\mathcal A$-vertices of $\Upsilon^*$.

\begin{theorem}\label{Realiz} Every admissible equipped graph $\Upsilon^*$ corresponds to an $\Omega$-stable flow $\phi^t\colon S\to S$ from $G$ on a closed surface $S$, besides:

(1) The Euler characteristic of $S$ can be calculated by the formula 
\begin{equation}\label{EjHar}
\chi(S)=\sum\limits_{\mathcal M}(X_{\mathcal M}-Y_{\mathcal M})+N_\mathcal A;
\end{equation}

(2) $S$ is orientable if and only if every four-colour graph equipping $\Upsilon^*$ has not cycles of an odd length and every $\mathcal L$-vertex is incident to exactly two edges.
\end{theorem}

\subsection{The algorithmic results}

An algorithm for solving the isomorphism problem
is considered to be \emph{efficient} if its working time is bounded by a polynomial on the length of the input data.
Algorithms of such kind are also called \emph{polynomial-time} or simply \emph{polynomial}. This commonly recognized
definition of efficient solvability rises to A. Cobham \cite{C64}. A common standard of intractability is
NP-completeness \cite{GJ79}. The complexity status of the
isomorphism problem is still unknown, i.e., for the class of all graphs,
neither its polynomial-time solvability nor its NP-completeness is proved at the moment. Fortunately, four-colour
graphs and directed graphs of flows are not graphs of the general type, as they can be embedded into a fixed surface on which
flows are defined, i.e. the ambient surface. That allows to prove the following theorems.

\begin{theorem} \label{isomorphism for equipped graphs} Isomorphism of the equipped graphs $\Upsilon^*_{\phi^t}$, $\Upsilon^*_{\phi'^t}$ of flows $\phi^t$, $\phi'^t\in G$ can be recognized in polynomial time.
\end{theorem}

\begin{theorem} \label{orientability and Euler characteristic} The orientability of the ambient surface $S$ for an $\Omega$-stable flow $\phi^t$ can be tested in a linear time and the Euler characteristic of $S$ can be determined in quadratic time by means of the equipped graph $\Upsilon^*_{\phi^t}$.
\end{theorem}

\section{The dynamics of a flow $f^t\in G$ without limit cycles on a surface $M$}

In this section everywhere below $f^t\in G$ is a flow without limit cycles on a closed surface $M$. We give proofs for the results from Subsection \ref{Mvertex} and other results about flows without limit cycles. A part of them was proved in \cite{KrPoMit}, \cite{KruMaPoMS} and \cite{grin} but we repeat them for a completeness. 

\subsection{General properties}
Firstly let us give a necessary proposition, which we will use for the proof of the classification theorem.

\begin{proposition}[\cite{grin}, Theorem 2.1.1] \label{Wupodmn}
$ $

1) $M=\bigcup\limits_{p\in\Omega(f^t)} W^u_p=\bigcup\limits_{p\in\Omega(f^t)} W^s_p$;

2) $W^u_p$ $(W^s_p)$ is a smooth submanifold of $M$ diffeomorphic to $\mathbb R^i$ $(\mathbb R^{2-i})$ for every fixed point $p\in\Omega^i_{f^t}$.
\end{proposition}

Let $p$ be a fixed point of $f^t$. Let us denote by $l^u_p$ ($l^s_p$) the unstable (stable) separatrix of $p$.

\begin{lemma}\label{SdljaS} For every sink $\omega$ (source $\alpha$) of $f^t$ there is at least one saddle point $\sigma$ with an unstable separatrix $l^u_{\sigma}$ (a stable separatrix $l^s_{\sigma}$) such that $cl(l^u_\sigma)\backslash (l^u_{\sigma})=\{\sigma,\omega\}$ $(cl(l^s_\sigma)\backslash (l^s_{\sigma})=\{\sigma,\alpha\})$.
\end{lemma}

\begin{proof} Supposing the contrary for some sink $\omega$, we get by the item 1) of Proposition \ref{Wupodmn} that $cl(W^s_\omega)=W^s_\omega\cup\bigcup\limits_{i=1}^k\{\alpha_i\}$, where $\alpha_i, i=\overline{1,k}$ is a source such that $W^u_{\alpha_i}\cap W^s_\omega\not=\emptyset$. Let us show that $W^u_{\alpha_i}\subset cl(W^s_\omega)$.

Let us assume the contrary. Then, by the item 1) of Proposition \ref{Wupodmn}, there is a point $p\in\Omega_{f^t}$ such that $p\not=\omega$ and $W^s_p\cap W^u_{\alpha_i}\not=\emptyset$. Let $x_\omega$ and $x_p$ be points such that $x_\omega\in W^u_{\alpha_i}\cap W^s_\omega$ and $x_p\in W^u_{\alpha_i}\cap W^s_p$. As the manifold $W^u_{\alpha_i}\backslash\{\alpha_i\}$ is homeomorphic to $\mathbb R^2\backslash\{O\}$ by the item 2) of Proposition \ref{Wupodmn}, then there is a simple path $c\colon [0,1]\to(W^u_{\alpha_i}\backslash\{\alpha_i\})$ connecting $x_\omega=c(0)$ with $x_p=c(1)$. Then, there is a value $\tau\in(0,1)$ such that $c(\tau)\notin W^s_\omega$ and $c(t)\in W^s_\omega$ for $t<\tau$. Consequently, there is a point $r\in\Omega_{f^t}$ such that $r\not=\omega$ and $c(\tau)\in W^s_r$. Besides, the point $c(\tau)\in cl(W^s_\omega)$. But if $c(\tau)\in cl(W^s_\omega)$, then $c(\tau)=\alpha_{i_0}$ for some $i_0\in\{1,\dots,k\}$ and, consequently, $\alpha_{i_0}\in W^u_{\alpha_i}$ that is the contradiction with the definition of the unstable manifold of a fixed point.

We have got that $W^u_{\alpha_i}\subset W^s_\omega$ for any $i=\overline{1,k}$ and, hence, the set $cl(W^s_\omega)$ is open as it contains every point with some open neighbourhood. As $cl(W^s_\omega)$ is both open and closed, $cl(W^s_\omega)=M$. Then $\Omega_{f^t}$ does not contain saddle points, that contradicts with conditions of the lemma.

The affirmation for sources can be proved by conversion from $f^t$ to $f^{-t}$. 
\end{proof}

\begin{lemma}\label{ZamykanieBezVsego}
Let $p$ be a fixed point of $f^t$. Then

$(i)$ if $p\in\Omega^1_{f^t}$, then 
\begin{equation*}
cl(l^u_p)\backslash(l^u_p\cup\{p\})=\begin{cases}
\{\sigma\}\subset\Omega^1_{f^t} &\text{and  $\,\,\,\,\,\,l^u_p=l^s_\sigma$,}\\
\{\omega\}\subset\Omega^0_{f^t} &\text{and $\,\,\,\,\,\,l^u_p\subset W^s_\omega$;}
\end{cases}
\end{equation*}

$(ii)$ if $p\in\Omega^2_{f^t}$, then
$cl(l^u_p)\backslash(l^u_p\cup\{p\})=\bigcup\limits_{\sigma\in\Omega_p}cl(l^u_\sigma)$, where $\Omega_p$ is a non-empty subset of $\Omega^1_{f^t}$.
\end{lemma}

\begin{proof}
Consider the case $(i)$, where $p$ is a saddle point. Let $x\in cl(l^u_p)$. Any point of $l^u_p$ is a point of $W^s_r$ for some fixed point $r$ by the item 1) of Proposition \ref{Wupodmn}. The point $r$ can be: a) a sink; b) a saddle point; c) a source.

a) Let us consider a sink $r=\omega$ such that $x\in W^s_\omega$. As $\omega$ is the source and $l^u_p=\mathcal O_x$, where $\mathcal O_x$ is the orbit of $x$, then $l^u_p\subset W^s_\omega$. So $cl(l^u_p)\backslash (l^u_p\cup\{p\})=\{\omega\}$.

b) Let us consider a saddle point $r=\sigma$ such that $x\in W^s_\sigma$. In this case $l^u_p=l^s_\sigma$. So $cl(l^u_p)\backslash(l^u_p\cup\{p\})=\{\sigma\}$.

c) Assume that there is a source $r=\alpha$ such that $x\in W^s_\alpha$. As $W^s_\alpha=\alpha$, then $\alpha\in l^u_p$, which is impossible because $l^u_p$ consists of wandering points. Consequently, the case c) is impossible.

Consider the case $(ii)$: $p=\alpha$ is a source.

The item 1) of Proposition \ref{Wupodmn} says that the set $A=cl(l^u_\alpha)\backslash(l^u_\alpha\cup\{\alpha\})$ is $f^t$-invariant subset of $W^u_{\Omega^1_{f^t}}\cup\Omega^0_{f^t}$. Then, for the proof of our Lemma all we need is to prove that a) if $\sigma\in A$ for some $\sigma\in\Omega^1_{f^t}$, then $l^u_\sigma\subset A$ and b) if $\omega\in A$ for some $\omega\in\Omega^0_{f^t}$, then there is $\sigma\in\Omega^1_{f^t}$ such that $\omega\in cl(l^u_\sigma)$ and $l^u_\sigma\subset A$.

a) As $\sigma\in A$, there is a sequence $x_n\in (l^u_\alpha\setminus W^s_\sigma)$ such that $x_n\to\sigma$ for $n\to\infty$. Then $\mathcal O_{x_n}\subset l^u_\alpha$ and, due to the known behaviour of our flow near $\sigma$ (see, Proposition \ref{LokSopr}), the set $\bigcup\limits_{n\in\mathbb N} \mathcal O_{x_n}$ contains in its closure the separatrix $l^u_\sigma$.

b) Let $\omega\in A$. According to Lemma \ref{SdljaS} there is a finite set of saddle points $\sigma_1,\dots,\sigma_k\in\Omega^1_{f^t}$ such that $\omega\in cl(l^u_{\sigma_i})$ for $i=\overline{1,k}$. Then the set $W^s_\omega\backslash(\omega\cup\bigcup\limits_{i=1}^kl^u_{\sigma_i})$ consists of a finite number of connected components, at least one of them belongs to $l^u_\alpha$, denote it by $Q$. Thus there is at least one saddle point $\sigma_{i_0},i_0\in\{1,\dots,k\}$ whose separatrices $l^u_{\sigma_{i_0}}$ belongs to $cl(Q)$. Thus $l^u_{\sigma_{i_0}}\subset A$.
\end{proof}

The statement similar to Lemma \ref{ZamykanieBezVsego} may be proved for the stable separatrices of the flow $f^t$.

\subsection{The proof for Lemma \ref{1ist1st}}

We remind that a cell $J$ is a connected component of the set $\tilde M=M\backslash (cl(W^u_{\Omega^1_{f^t}})\cup cl(W^s_{\Omega^1_{f^t}}))$.

Due to Proposition \ref{Wupodmn} 
\begin{equation*}
\tilde M=(\bigcup\limits_{\alpha\in\Omega^2_{f^t}} l^u_\alpha)\backslash (\bigcup\limits_{\sigma\in\Omega^1_{f^t}} l^s_\sigma).
\end{equation*} 
Then every connected component $J$ of $\tilde M$ is a subset of $l^u_\alpha$ for a source $\alpha$. Similarly  
\begin{equation*}
\tilde M=(\bigcup\limits_{\omega\in\Omega^0_{f^t}} l^s_\omega)\backslash (\bigcup\limits_{\sigma\in\Omega^1_{f^t}} l^u_\alpha).
\end{equation*} 
Then every connected component $J$ of $\tilde M$ is a subset of $l^s_\omega$ for a sink $\omega$. Thus 
\begin{equation*}
J\subset(W^u_\alpha\cap W^s_\omega)
\end{equation*} 
and, consequently, the cell $J$ is a union of trajectories going from $\alpha$ to $\omega$.

\begin{flushright}$\square$\end{flushright}

\subsection{The proof for Lemma \ref{vidmnob}}

We remind that we choose a one trajectory $\theta_J$ in the cell $J$ and called it by a $t$-curve. Also we defined $\mathcal T=\bigcup\limits_{J\subset\tilde M}\theta_J,\,\,\bar M=\tilde{M}\backslash\mathcal T$. Besides, we called by a $c$-curve a separatrix connecting saddle points (``connection''), by a $u$-curve an unstable saddle separatrix with a sink in its closure, by a $s$-curve a stable saddle separatrix with a source in its closure. A {\it polygonal region $\Delta$} is the closure of a connecting component of $\bar M$.

Due to Lemma \ref{1ist1st} every cell $J$ belongs to the basin of the source $\alpha$ and, due to Lemma \ref{ZamykanieBezVsego}, $J$ is situated in $W^u_\alpha$ between too (may be coincident) $s$-curves. A polygonal region $\Delta$ can be created by removal a $t$-curve from $J$. As $W^u_\alpha$ is homeomorphic to $\mathbb R^2$, due to Proposiition \ref{Wupodmn}, then $\Delta$ is homeomorphic to a sector in $\mathbb R^2$, i.e. $\Delta$ is homeomorphic to an open disk. By construction, the boundary of $\Delta$ contains unique $s$-curve and unique $t$-curve. As $\Delta$ belongs to the basin of the sink $\omega$ in the same time, then it is restricted by unique $u$-curve. By $(ii)$ of Lemma \ref{ZamykanieBezVsego} the region $\Delta$ is restricted by a finite number of $c$-curves. We have got that the only possible structure of the boundary of a polygonal region $\Delta$ is the structure depicted on Figure \ref{MnogObl} up to a number of the $c$-curves.

\begin{flushright}$\square$\end{flushright}

\subsection{The proof of Lemma \ref{GmAdm}}

We remind that $\pi_{f^t}$ is the one-to-one correspondence between polygonal regions and vertices, also between colour curves of $f^t$ and colour edges of $\Gamma_{\mathcal M}$ respectively.

As $f^t$ given on the surface $M$ and every vertex of $\Gamma_{\mathcal M}$ corresponds to some its polygonal region, then, we can create a graph isomorphic to $\Gamma_{\mathcal M}$ with each vertex in its own polygonal region and with edges that are curves embedded in $M$, joining the vertices and crossing each its side at the unique point. Such graph is obviously isomorphic to $\Gamma_{\mathcal M}$. Therefore, without loss of generality, let's mean that $\Gamma_{\mathcal M}$ is embedded in $M$. As every polygonal region side adjoins to exactly two different polygonal regions, then $\Gamma_{\mathcal M}$ has not cycles of length $2$, i.e. $\Gamma_{\mathcal M}$ is simple one.

As to each point $p\in\Omega_{f^t}$ a finite number of polygonal regions divided by colour curves adjoins, then the point $p$ by means $\pi_{f^t}$ one-to-one corresponds to a cycle of the vertices corresponding to the regions adjoining to $p$ and of the colour edges crossing colour curves exiting out of $p$. So exactly $4$ polygonal regions divided by $u$-, $s$- or $c$-curves adjoin to a saddle point. If to mean $u$- and $s$-edges as nominal $c$-edges, we get that every saddle point corresponds to the $c^*$-cycle of $\Gamma_{\mathcal M}$. Conversely also is correct, because every $c^*$-cycle can be placed in a neighbourhood of the single saddle point so that such neighbourhoods of different $c^*$-cycles doesn't cross one another. In this way $\Gamma_{\mathcal M}$ contains $c^*$-cycles and each such cycle has length $4$. Consequently $\Gamma_{\mathcal M}$ is admissible.

\begin{flushright}$\square$\end{flushright}

\subsection{The proof of Proposition \ref{TochkiICikly}}

The correspondence between $\Omega^1_{f^t}$ and the set of $c^*$-cycles follows from the proof of Lemma \ref{GmAdm}. The basin of every sink  $\omega$ is divided by $u$- and $t$-curves alternately lying in $W^s_\omega$. Consequently $\omega$ corresponds to unique $tu$-cycle of $\Gamma_{\mathcal M}$ by means $\pi_{f^t}$. Conversely it is also corrected because as basins of different sinks are divided by $s$ and $c$-curves then each $tu$-cycle can be situated in the basin of the unique sink. In this way $\pi_{f^t}$ creates one-to-one corresponding between the set $\Omega^0_{f^t}$ and the set of $tu$-cycles. The correspondence between $\Omega^2_{f^t}$ and the set of $st$-cycles can be proved similarly.

\begin{flushright}$\square$\end{flushright}

\section{The proof for the classification Theorem \ref{Classif}}

In this section we consider $\Omega$-stable flow $\phi^t\in G$ on closed surface $S$ and prove that the isomorphic class of its equipped graph $\Upsilon^*_{\phi^t}$ is a complete topological invariant. 

\subsection{The necessary condition of Theorem \ref{Classif}}
Let two $\Omega$-stable flows $\phi^t,\,\phi'^t\in G$ given on a closed surface $S$ be topological equivalent, i.e. there is a homeomorphism $h\colon S\to S$ mapping trajectories of $\phi^t$ to trajectories of $\phi'^t$. Let us think without loss of generality that the cutting set $\mathcal R'$ of $\phi'^t$ is created so that $\mathcal R'=h(\mathcal R)$, where $\mathcal R$ is the cutting set of $\phi^t$. Also we can think that the restriction $\mathcal T'$ of the set of $t$-curves of $\phi'^t$ to the $\mathcal M$-regions of $\phi'^t$ is created so that $\mathcal T'=h(\mathcal T)$, where $\mathcal T$ is the restriction of the set of $t$-curves of $\phi^t$ to the $\mathcal M$-regions of $\phi^t$. Then $h$ maps the elementary and the polygonal regions of $\phi^t$ to the elementary and the polygonal regions of $\phi'^t$ respectively. 

Recall that $\pi^*_{\phi^t}$ is the one-to-one correspondence between the elementary regions and the vertices, the cutting circles and the edges, the directions of the trajectories and the directions of the edges, the consistencies of the orientations of the limit circles for the $\mathcal E$-regions and the weights of the $\mathcal E$-vertices, the $\mathcal M$-regions and the four-colour graphs, the stable limit cycles and the $tu$-cycles, the unstable limit cycles and the $st$-cycles respectively. 
Let us define the isomorphism $\xi\colon\Upsilon^*_{\phi^t}\to\Upsilon^*_{\phi'^t}$ by the formula 
\begin{equation*}
\xi=\pi^*_{\phi'^t}h(\pi^*_{\phi^t})^{-1}.
\end{equation*} 
As $h$ carries out the topological equivalence of $\phi^t$ and $\phi'^t$ then it preserves the types of elementary regions and, hence, $\xi$ preserves the types of the vertices. As $h$ preserves the orientation on the trajectories then the weights of vertices $\mathcal E$ and $\xi(\mathcal E)$ are equal. 

Let $\Gamma_{\mathcal M}$ is the four-colour graph for some vertex $\mathcal M$, $\Gamma_{\xi(\mathcal M)}$ is the four-colour graph for the vertex $\mathcal M'=\xi(\mathcal M)$. Recall that $\phi^t|_{\mathcal M}=f^t|_{\mathcal M}$, $\phi'^t|_{\mathcal M'}=f'^t|_{\mathcal M'}$ and $\pi_{f^t}$, $\pi_{f'^t}$ is the one-to-one correspondence between the polygonal regions and the vertices, also between the colour curves of $f^t$, $f'^t$ and the colour edges of the four-colour graph $\Gamma_{\mathcal M}$, $\Gamma_{\mathcal M'}$ respectively.  

As $\Gamma_{\mathcal M}$ is the four-colour graph of the region $\mathcal M$, then $\Gamma_{\mathcal M}=\pi^*_{\phi^t}(\mathcal M)$. Let $\Gamma_{\mathcal M'}=\Gamma_{\xi(\mathcal M)}=\pi^*_{\phi'^t}(h(\mathcal M))$. As $h$ maps the polygonal regions of $f^t$ to the polygonal regions of $f'^t$, then there exists the isomorphism $\psi\colon\Gamma_{\mathcal M}\to\Gamma_{\mathcal M'}$ defined by the formula 
\begin{equation*}
\psi_{\mathcal M}=\pi_{f'^t}h\pi^{-1}_{f^t}.
\end{equation*} 
As $\mathcal R'=h(\mathcal R)$, then $\psi(\tau_{_{\mathcal M,\mathcal L}})=\tau_{_{\xi(\mathcal M),\xi(\mathcal L)}}$ and the orientations of $\psi_{\mathcal M}(\tau_{_{\mathcal M, \mathcal L}})$ and $\tau_{_{\xi(\mathcal M),\xi(\mathcal L)}}$ coincide (similarly for $\tau_{_{\mathcal L, \mathcal M}}$). Thus $\xi$ is the required isomorphism. 

\subsection{The sufficient condition of Theorem \ref{Classif}}
Let graphs $\Upsilon^*_{\phi^t}$ and $\Upsilon^*_{\phi'^t}$ be isomorphic by means of $\xi$. To prove the topological equivalence of the flows we need to create homeomorphisms between elementary regions  mapping the trajectories of $\phi^t$ to the trajectories of $\phi'^t$ so that for two elementary regions the homeomorphisms on their common boundaries coincide.

{\bf I. $\mathcal M$-region.} Let us consider some $\mathcal M$-region of the flow $\phi^t$. Consider the region 
\begin{equation*}
\mathcal M'=(\pi^*_{\phi'^t})^{-1}\xi\pi^*_{\phi^t}(\mathcal M)
\end{equation*} 
of the flow $\phi'^t$. Their four-colour graphs $\Gamma_{\mathcal M}$ and $\Gamma_{\mathcal M'}$ are isomorphic by means of $\psi$. Let $f^t\colon M\to M$ ($f'^t\colon M'\to M'$) be the flow corresponding to $\Gamma_{\mathcal M}$ ($\Gamma_{\mathcal M'}$). Recall that in Subsection \ref{Mvertex} we defined the flow $f^t$ ($f'^t$) on surface $M$ ($M'$) such that $M\cap S=\mathcal M$ ($M'\cap S=\mathcal M'$) and $\phi^t|_{\mathcal M}=f^t|_{\mathcal M}$ ($\phi'^t|_{\mathcal M'}=f'^t|_{\mathcal M'}$).

Consider a polygonal region $\Delta\in\Delta_{f^t}$. The $\Delta$'s boundary contains an unique source $\alpha$, an unique sink $\omega$ and $n$ saddle points $\sigma_1,\sigma_2,\dots,\sigma_n,\,n\in\mathbb N$, and the saddle points are ordered so that their labels increase while moving along the $\Delta$'s boundary according to the direction from the source to the sink on the $t$-curve. Consider the polygonal region $\Delta'\in\Delta_{f'^t}$ such that 
\begin{equation*}
\Delta'=\pi^{-1}_{f'^t}\psi\pi_{f^t}(\Delta).
\end{equation*} 
The isomorphism $\psi$ provides an equal number of the same-colour edges exiting out of graph vertices corresponding to $\Delta$ and $\Delta'$. It implies that $\Delta'$'s boundary contains exactly an unique source $\alpha'$, an unique sink $\omega'$ and $n$ saddle points $\sigma'_1,\sigma'_2,\dots\sigma'_n$ ordered so that their labels increase while moving along the $\Delta'$'s boundary according to the direction from the source to the sink on the $t$-curve. As $\psi$ preserves the colours of the edges and the numbers of the $c$-edges, then for creation of the homeomorphism $h_M\colon M\to M'$ doing the topological equivalence of $f^t$ and $f'^t$ all we need is to construct the homeomorphism $h_\Delta\colon cl(\Delta)\to cl(\Delta')$ mapping the trajectories of $f^t$ from $cl(\Delta)$ to the trajectories of $f'^t$ from $cl(\Delta')$ so that 
\begin{equation*}
h_\Delta|_{cl(\Delta)\cap cl(\tilde\Delta)}=h_{\tilde\Delta}|_{cl(\Delta)\cap cl(\tilde\Delta)}
\end{equation*} 
for any polygonal regions $\Delta$, $\tilde\Delta$ of $f^t$.

{\bf Step 1.} Let us construct $h_\Delta$ in neighbourhoods of the node points. Let 
\begin{equation*}
u=\{(x,y)\in\mathbb R^2: x^2+y^2<1\}
\end{equation*} 
and recall that $a^t\colon\mathbb R^2\to\mathbb R^2$, $c^t\colon\mathbb R^2\to\mathbb R^2$ are the flows given by the formulas $a^t(x,y)=(2^{-t}x,2^{-t}y)$, $c^t(x,y)=(2^t x,2^t y)$ with the origin $O$ as a sink and a source point accordingly. By Proposition \ref{LokSopr} there exist the neighbourhoods $u_\omega$, $u_\alpha$ ($u_{\omega'}$, $u_{\alpha'}$) of $\omega$, $\alpha$ ($\omega'$, $\alpha'$) accordingly such that $f^t|_{u_\omega}$, $f^t|_{u_\alpha}$ ($f'^t|_{u_\omega}$, $f'^t|_{u_\alpha}$) are topologically conjugate to $a^t(x,y)|_u$, $c^t(x,y)|_u$ by means of some homeomorphisms $h_{\omega}\colon u_\omega\to u$, $h_{\alpha}\colon u_\alpha\to u$ ($h_{\omega'}\colon u_{\omega'}\to u$,  $h_{\alpha'}\colon u_{\alpha'}\to u$) accordingly. Without loss of generality let us think that these neighbourhoods do not cross each other for all polygonal regions.

For $r\in(0,1]$ let $S_r=\{(x,y)\in\mathbb R^2: x^2+y^2=r\}$ and $S_r^\omega=h_\omega^{-1}(S_r)$, $S_r^\alpha=h_\alpha^{-1}(S_r)$ ($S_r^{\omega'}=h_{\omega'}^{-1}(S_r)$, $S_r^{\alpha'}=h_{\alpha'}^{-1}(S_r)$).
Let $\{A\}=S^\omega_1\cap l_{\alpha,\omega}$, $\{A_0\}=S^\omega_1\cap l_{\omega,\sigma_1}$ ($\{A'\}=S_1^{\omega'}\cap l_{\alpha',\omega'}$, $\{A'_0\}=S^{\omega'}_1\cap l_{\omega',\sigma'_1}$) and $\{C\}=S_1^{\alpha}\cap l_{\alpha,\omega}$, $\{C_0\}=S_1^{\alpha}\cap l_{\alpha,\sigma_n}$ ($\{C'\}=S_1^{\alpha'}\cap l_{\alpha',\omega'}$, $\{C'_0\}=S_1^{\alpha'}\cap l_{\alpha',\sigma'_n}$).

Everywhere below we will denote by $m_{a,b}$ the closure of a segment of a cross-section to the trajectories of $f^t$ ($f'^t$) bounded by points $a$  ($a'$) and $b$ ($b'$). In particular denote by $m_{A,A_0}$ ($m_{A',A'_0}$) the segment which is the intersection $S^\omega_1\cap\Delta$ ($S^{\omega'}_1\cap\Delta'$) (see Figure \ref{MnogOblSOkr}). Let $x\in m_{A,A_0}$. Let $\mu_{A,A_0}\colon m_{A,A_0}\to[0,1]$ ($\mu_{A',A'_0}\colon m_{A',A'_0}\to[0,1]$) be the arbitrary homeomorphism such that $\mu_{{A,A_0}}(A)=0$ ($\mu_{A',A'_0}(A')=0$). Let 
\begin{equation*}
h_{m_{A,A_0}}=\mu^{-1}_{A'A'_0}\mu_{A,A_0}\colon m_{A,A_0}\to m_{A',A'_0}.
\end{equation*} 
Let $x\in m_{A,A_0}$ ($x'\in m_{A',A'_0}$) and $\mathcal O_{x}$ ($\mathcal O_{x'}$) be the trajectory of $x$ ($x'$). Let $x^\omega\in (cl(u_\omega)\cap\Delta\setminus\{\omega\})$, then $x^\omega= S_{r}^\omega\cap\mathcal O_{x}$ for some $r\in(0,1]$ and $x\in m_{A,A_0}$. Let us define the homeomorphism $h_{u_\omega}\colon cl(u_\omega)\cap\Delta\to cl(u_{\omega'})\cap\Delta'$ so that $h_{u_{\omega}}(\omega)={\omega'}$ and $h_{u_{\omega}}(x^\omega)=x'^{\omega'}$, where $x'^{\omega'}=S_{r}^{\omega'}\cap\mathcal O_{h_{m_{A,A_0}}(x)}$. Similarly for points $x^\alpha\in (cl(u_\alpha)\cap\Delta\setminus\{\alpha\})$ being the intersection point $x^\alpha=S_{r}^\alpha\cap\mathcal O_{x}$ for some $r\in(0,1]$ and $x\in m_{A,A_0}$, define the homeomorphism $h_{u_\alpha}\colon cl(u_\alpha)\cap\Delta\to cl(u_{\alpha'})\cap\Delta'$ so that $h_{u_{\alpha}}(\alpha)={\alpha'}$ and $h_{u_{\alpha}}(x^\alpha)=x'^{\alpha'}$, where $x'^{\alpha'}=S_{r}^{\alpha'}\cap\mathcal O_{h_{m_{A,A_0}}(x)}$. 

\begin{figure}[htb]
\centerline{\includegraphics [width=15 cm] {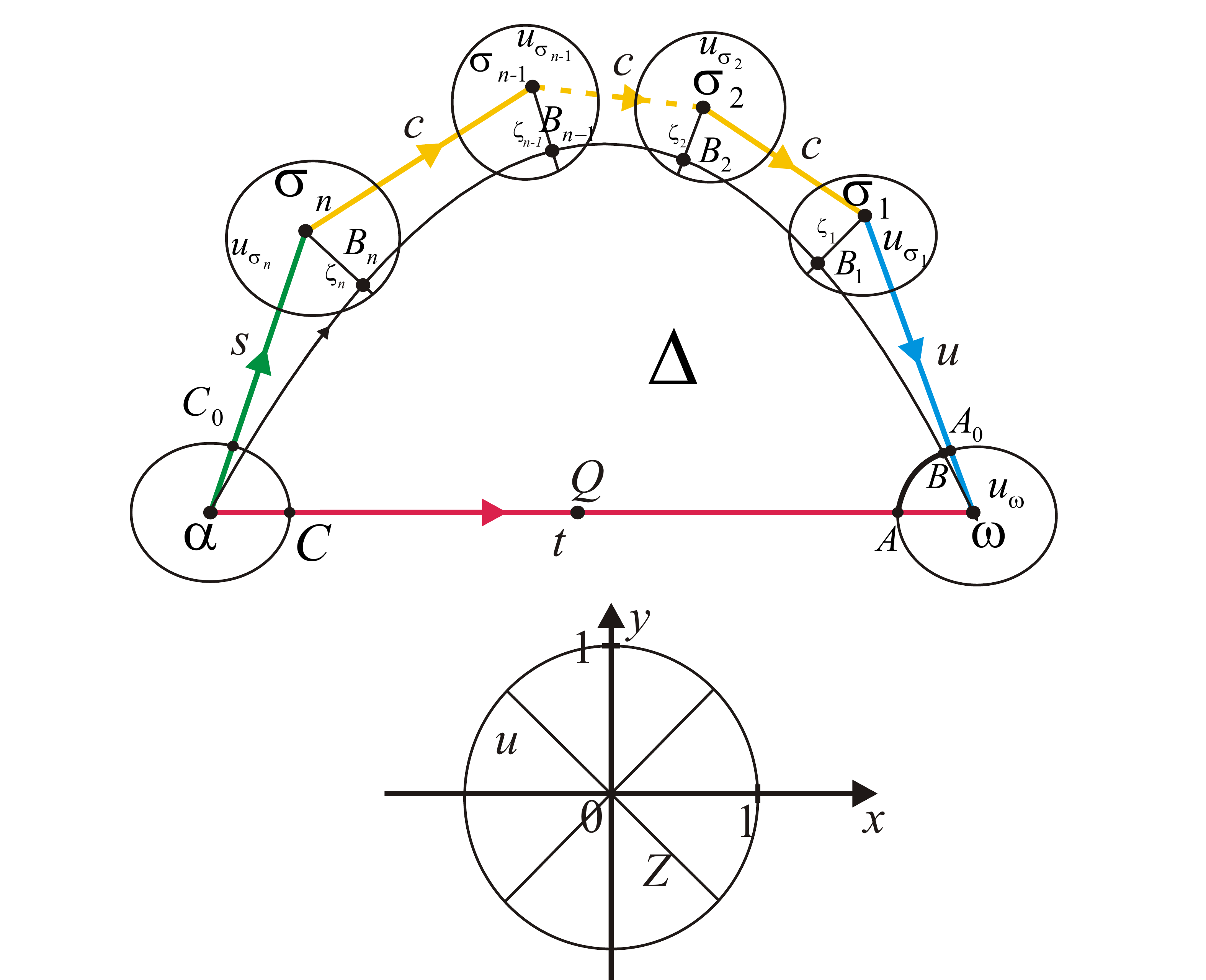}}
\caption{Steps 1-3: the construction of $h_\Delta$ in the neighbourhoods of the node points; the constructions of $\zeta_i$}\label{MnogOblSOkr}
\end{figure}

{\bf Step 2.} Let us construct $h_\Delta$ on the boundary of $\Delta$.

Everywhere below we will denote by $l_{a,b}$ the closure of a segment of a trajectory or an separatrix of a saddle point bounded by points $a$ and $b$, and by $\lambda_{a,b}$ we will denote its length. Notice that $l_{a,b}=l_{b,a}$ and $\lambda_{a,b}=\lambda_{b,a}$. For smooth segments $l_{a,b}$, $l_{a',b'}$ of trajectories of $f^t$, $f'^t$ we will call {\it a homeomorphism by the length of arc} a homeomorphism $h_{l_{a,b}}\colon l_{a,b}\to l_{a',b'}$ defined by the following rule for a point $x\in l_{a,b}$:  
\begin{equation*}
\lambda_{a',h_{l_{a,b}}(x)}=\frac{\lambda_{a,x}\cdot\lambda_{a',b'}}{\lambda_{a,b}}.
\end{equation*}

Thus, we construct the following homeomorphisms:  $h_{l_{A,C}}\colon l_{A,C}\to l_{A',C'}$, $h_{l_{A_0,\sigma_1}}\colon l_{A_0,\sigma_1}\to l_{A'_0,\sigma'_1}$, $h_{l_{C_0,\sigma_n}}\colon l_{C_0,\sigma_n}\to l_{C'_0,\sigma'_n}$ and $h_{l_{\sigma_i,\sigma_{i+1}}}\colon l_{\sigma_i,\sigma_{i+1}}\to l_{\sigma'_i,\sigma'_{i+1}}$.

A similar construction on the boundaries of all polygonal regions will provided $h_\Delta|_{cl(\Delta)\cap cl(\tilde\Delta)}=h_{\tilde\Delta}|_{cl(\Delta)\cap cl(\tilde\Delta)}$ for any polygonal regions $\Delta$, $\tilde\Delta$ of $f^t$.

{\bf Step 3.} Let us construct cross-sections connecting the saddle points with some point inside $l_{A,C}$ ($l_{A',C'}$).  

Let $Q\in int(l_{A,C})$ and $Q'=h_{l_{A,C}}(Q)$. Let us construct cross-sections $m_{Q,\sigma_1},\dots,m_{Q,\sigma_n}$ ($m_{Q',\sigma'_1},\dots,m_{Q',\sigma'_n}$) the following way.  

Recall that for $i=\overline{1,n}$ there exists a neighbourhood $u_{\sigma_i}$ ($u_{\sigma'_i}$) of $\sigma_i$ ($\sigma'_i$) such that $f^t|_{u_{\sigma_i}}$ ($f'^t|_{u_{\sigma_i}}$) is topologically conjugate to $b^t|_u$ by means of some homeomorphism $h_{\sigma_i}\colon u_{\sigma_i}\to u$ ($h_{\sigma'_i}\colon u_{\sigma'_i}\to u$), where $b^t\colon \mathbb R^2\to \mathbb R^2$ and $b^t(x,y)=\left(2^{-t}x,2^ty\right)$. Let 
\begin{equation*}
Z=\{(x,y)\in\mathbb R^2 \mid |x|=|y|\}\cap u.
\end{equation*}
The set $Z$ consists of the two intervals crossing in the origin and transversal to the trajectories of $b^t$. Let $\{\zeta_i\}=h^{-1}_{\sigma_i}(Z)\cap\Delta$ ($\{\zeta'_i\}=h^{-1}_{\sigma'_i}(Z)\cap\Delta'$). Let us choose the point $B\in m_{A,A_0}$ such that  $\mathcal O_B\cap\zeta_i\not=\emptyset$ and $\mathcal O_{B'}\cap\zeta'_i\not=\emptyset$ for $i=\overline{1,n}$, where $B'=h_{m_{A,A_0}}(B)$ (see Figure \ref{Sekuschie}). 

Let $\{B_i\}=\mathcal O_B\cap\zeta_i$ ($\{B'_i\}=\mathcal O_{B'}\cap\zeta'_i$). Denote by $m_{B_i,\sigma_i}$ ($m_{B'_i,\sigma'_i}$) the subset $\zeta_i$ ($\zeta'_i$) bounded by $B_i$ ($B'_i$) and $\sigma_i$ ($\sigma'_i$). Let $t_0\in\mathbb R$ ($t'_0\in\mathbb R$) and $t_i\in\mathbb R$ ($t'_i\in\mathbb R$) be such that $A=f^{t_0}(Q)$ ($A'=f^{t'_0}(Q')$) and $B=f^{t_i}(B_i)$ ($B'=f^{t'_i}(B'_i)$). Let 
\begin{align*}
m_{B_i,Q}=&\left\{f^{\frac{(-t_i)\cdot\mu_{A,A_0}(x)+(-t_0)(\mu_{A,A_0}(B)-\mu_{A,A_0}(x))}{\mu_{A,A_0}(B)}}(x),\,x\in m_{A,A_0}\right\}, \\
m_{B'_i,Q'}=&\left\{f'^{\frac{(-t'_i)\cdot\mu_{A',A'_0}(x')+(-t'_0)(\mu_{A',A'_0}(B')-\mu_{A',A'_0}(x'))}{\mu_{A',A'_0}(B')}}(x'),\,x'\in m_{A',A'_0}\right\}.
\end{align*}
Then $m_{Q,\sigma_i}=m_{Q,B_i}\cup m_{B_i,\sigma_i}$ ($m_{Q',\sigma'_i}=m_{Q',B'_i}\cup m_{B'_i,\sigma'_i}$) (see Figure \ref{Sekuschie}).

\begin{figure}[htb]
\centerline{\includegraphics [width=10.5 cm] {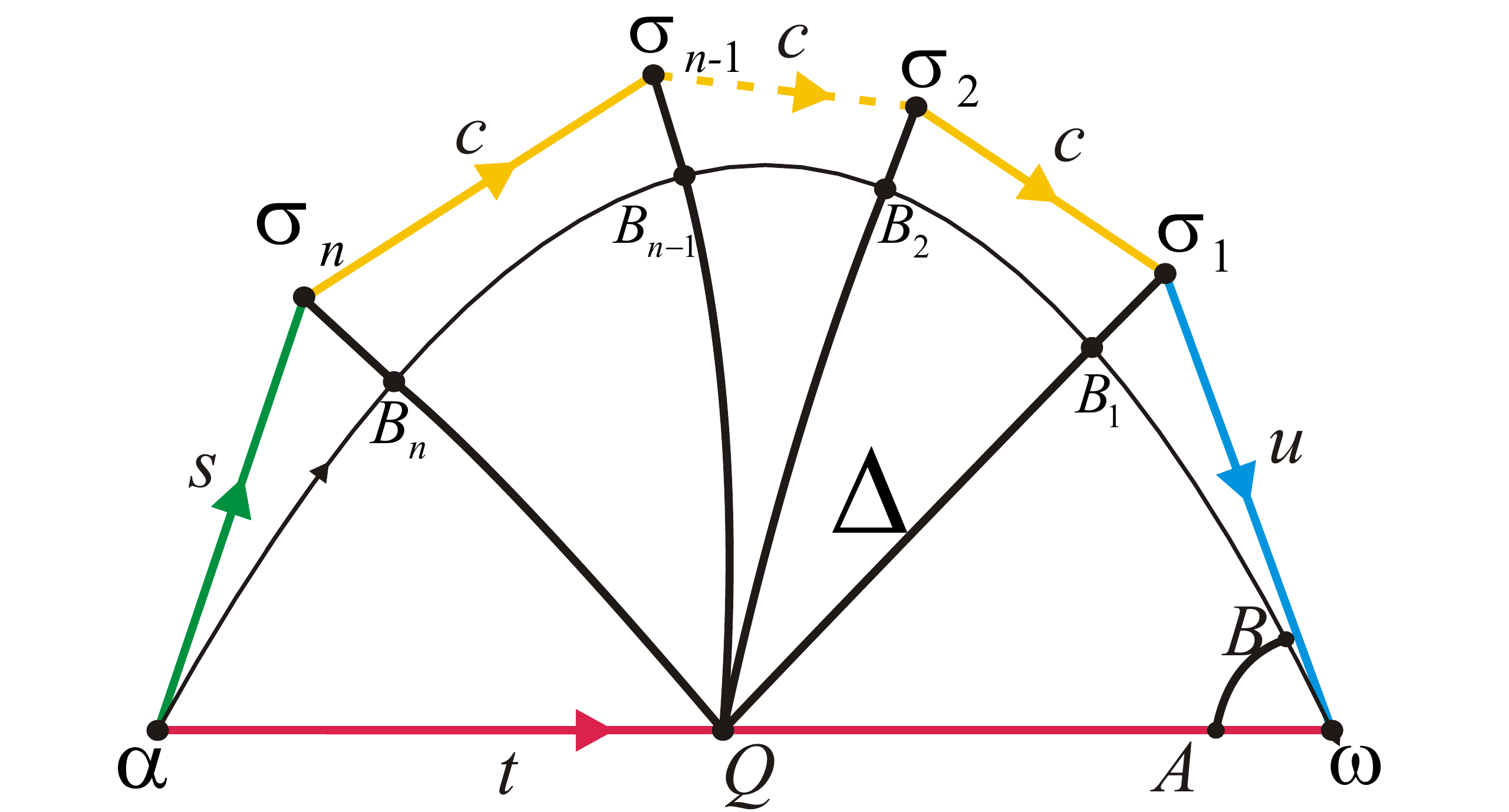}}
\caption{Construction of the cross-sections $m_{Q,\sigma_1},\dots,m_{Q,\sigma_n}$}\label{Sekuschie}
\end{figure}

{\bf Step 4.} Let us continue $h_\Delta$ inside $\Delta$.

Let $x_0\in m_{A,A_0}$, $x'_0=h_{m_{A,A_0}}(x_0)$, $\mathcal O_{x_0}$ is the trajectory of $x_0$ and $\mathcal O_{x'_0}$ is the trajectory of $x'_0$. Let $\{x_i\}=\mathcal O_{x_0}\cap m_{Q,\sigma_i}$, $\{x'_i\}=\mathcal O_{x'_0}\cap m_{Q',\sigma'_i}$ for $i=\overline{1,n}$, $\{x_{n+1}\}=\mathcal O_{x_0}\cap m_{C,C_0}$, $\{x'_{n+1}\}=\mathcal O_{x'_0}\cap m_{C',C'_0}$ (see Figure \ref{Vnutr}).

\begin{figure}[htb]
\centerline{\includegraphics [width=10.5 cm] {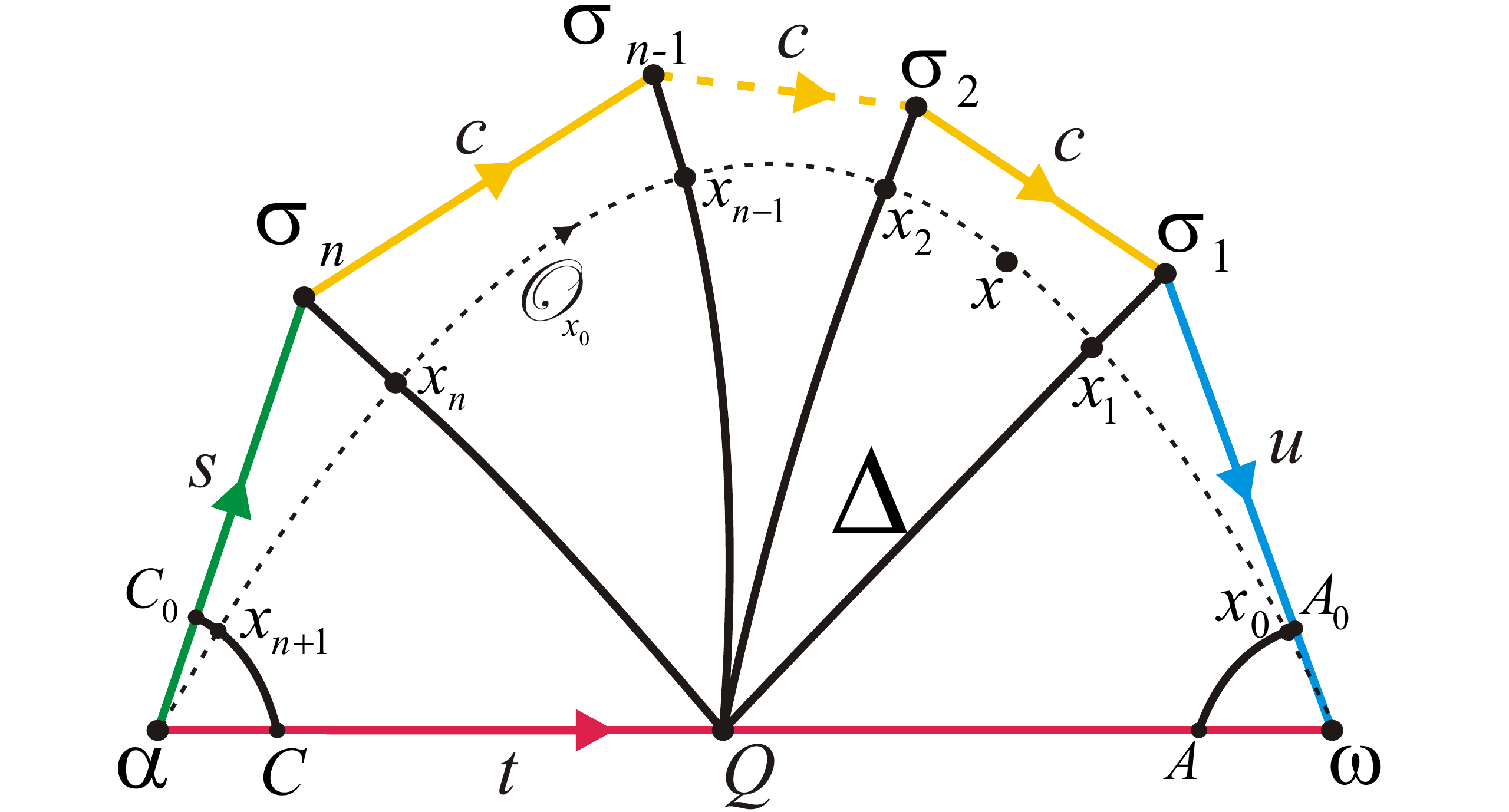}}
\caption{Extention of $h_\Delta$ inside  $\Delta$}\label{Vnutr}
\end{figure}

Extend $h_\Delta$ onto the trajectory $\mathcal O_{x_0}$ so that $h_\Delta|_{l_{x_i,x_{i+1}}}=h_{l_{x_i,x_{i+1}}}\colon l_{x_i,x_{i+1}}\to l_{x'_i,x'_{i+1}}$.

Denote by $h_M\colon M\to M$ a homeomorphism composed by $h_\Delta,\Delta\in\Delta_{f^t}$. As $M\cap S=\mathcal M$ and $M'\cap S=\mathcal M'$, then we define the homeomorphism $h_{\mathcal M}\colon cl(\mathcal M)\to cl(\mathcal M')$ by the formula 
\begin{equation*}
h_{\mathcal M}=h_M|_{cl(\mathcal M)}.
\end{equation*} 
Thus we have the homeomorphism 
\begin{equation*}
h_{\mathcal M}\colon cl(\mathcal M)\to cl(\mathcal M')
\end{equation*} 
for every $\mathcal M$-region of the flow $\phi^t$.

{\bf II. $\mathcal E$-region.} Let us consider some $\mathcal E$-region of the flow $\phi^t$. Consider the $\mathcal E'$-region of the flow $\phi'^t$ such that 
\begin{equation*}
\mathcal E'=(\pi^*_{\phi'^t})^{-1}\xi\pi^*_{\phi^t}(\mathcal E).
\end{equation*}  
These two regions are of the same type because of the weight of the vertices corresponding to them. 

Let $E_1$ and $E_2$ be the connected components of $\partial\mathcal E$. Then they are cutting circles and, hence, $E'_i=(\pi^*_{\phi'^t})^{-1}\xi\pi^*_{\phi^t}(E_i),\,i=1,2$ are cutting circles which are the connected components of $\partial\mathcal E'$.

Let $h_{E_1}\colon E_1\to E'_1$ be an arbitrary homeomorphism preserving orientations of $E_1$ and $E'_1$. 
Let $x_0\in E_1$ and  $\{x_1\}=\mathcal O_{x_0}\cap E_2$. Let $x'_0=h_{E_1}(x_0)$ and $\{x'_1\}=\mathcal O_{x'_0}\cap E'_2$. Let us construct the homeomorphism $h_{\mathcal E}\colon cl(\mathcal E)\to cl(\mathcal E')$ so that $h_{\mathcal E}|_{l_{x_0,x_1}}=h_{l_{x_0,x_1}}\colon l_{x_0,x_1}\to l_{x'_0,x'_1}$. 

Thus we have the homeomorphism 
\begin{equation*}
h_{\mathcal E}\colon cl(\mathcal E)\to cl(\mathcal E')
\end{equation*} 
for every $\mathcal E$-region of the flow $\phi^t$.

{\bf III. $\mathcal A$-region.} Let us consider some $\mathcal A$-region of the flow $\phi^t$ with a source $\alpha$ (for definiteness) inside. Consider the region 
\begin{equation*}
(\pi^*_{\phi'^t})^{-1}\xi\pi^*_{\phi^t}(\mathcal A)
\end{equation*}
of the flow $\phi'^t$. We perfectly know that it is the $\mathcal A'$-region with a source inside because of directions of edges. 

$\mathcal A$ ($\mathcal A'$) is surely surrounded by some $\mathcal L$, ($\mathcal L'=(\pi^*_{\phi'^t})^{-1}\xi\pi^*_{\phi^t}(\mathcal L)$).

Recall that 
\begin{equation*}
u=\{(x,y)\in\mathbb R^2\mid x^2+y^2<1\}.
\end{equation*}
Due to Proposition \ref{LokSopr} the source $\alpha$ ($\alpha'$) has a neighbourhood $u_\alpha$ ($u_{\alpha'}$) and the homeomorphism $h_\alpha\colon u_\alpha\to u$ ($h_{\alpha'}\colon u_{\alpha'}\to u$) such that $\phi^t|_{u_\alpha}$ ($\phi'^t|_{u_{\alpha'}}$) is conjugate with $c^t|_u$. Also recall that $S_r=\{(x,y)\in\mathbb R^2: x^2+y^2=r\}$ for $r\in(0,1]$ and $S_r^\alpha=h_\alpha^{-1}(S_r)$ ($S_r^{\alpha'}=h_{\alpha'}^{-1}(S_r)$). Notice that $S^\alpha_1=\partial u_\alpha$ ($S^{\alpha'}_1=\partial u_{\alpha'}$)

We know that $\partial\mathcal A$ ($\partial\mathcal A'$) is oriented. Then now we orient $\partial u_\alpha=S^\alpha_1$ ($\partial u_{\alpha'}=S^{\alpha'}_1$) consistently with $\partial\mathcal A$ ($\partial\mathcal A'$). Let $h_{S^\alpha_1}\colon S^\alpha_1\to S^{\alpha'}_1$ be the arbitrary homeomorphism preserving orientations of $S^\alpha_1$ and $S^{\alpha'}_1$. Let $x\in S^\alpha_1$ ($x'\in S^{\alpha'}_1$) and $\mathcal O_x$ ($\mathcal O_{x'}$) is the trajectory of $x$ ($x'$). Let $x^\alpha\in (cl(u_\alpha)\setminus\{\omega\})$, then $x^\alpha= S_{r}^\alpha\cap\mathcal O_{x}$ for some $r\in(0,1]$ and $x\in S_1^\alpha$. Let us define the homeomorphism $h_{u_\alpha}\colon cl(u_\alpha)\to cl(u_{\alpha'})$ so that $h_{u_{\alpha}}(\alpha)={\alpha'}$ and $h_{u_{\alpha}}(x^\alpha)=x'^{\alpha'}$, where $x'^{\alpha'}=S_{r}^{\alpha'}\cap\mathcal O_{h_{S^\alpha_1}(x)}$.

Let $x_0\in\partial u_\alpha$, $x'_0\in\partial u_{\alpha'}$ and $x'_0=h_{u_\alpha}(x_0)$. Let $\mathcal O_{x_0}$ ($\mathcal O_{x'_0}$) be the trajectory of $x_0$ ($x'_0$) and $\{x_1\}=\mathcal O_{x_0}\cap\partial A$ ($\{x'_1\}=\mathcal O_{x'_0}\cap\partial A'$). Let us define the homeomorphism $h_{cl(\mathcal A)\backslash u_\alpha}\colon cl(\mathcal A)\backslash u_\alpha\to cl(\mathcal A')\backslash u_{\alpha'}$ so that $h_{cl(\mathcal A)\backslash u_\alpha}|_{l_{x_0,x_1}}=h_{l_{x_0,x_1}}\colon l_{x_0,x_1}\to l_{x'_0,x'_1}$ for any $x_0\in\partial u_\alpha$.

So we define the homeomorphism $h_{\mathcal A}\colon\ cl(\mathcal A)\to cl(\mathcal A')$ by the formula
\begin{equation*}
h_{\mathcal A}(x)=\begin{cases}
h_{u_\alpha}(x) &\text{if $x\in u_\alpha$,}\\
h_{cl(\mathcal A)\backslash u_\alpha}(x)  &\text{if $x\in cl(\mathcal A)\backslash u_\alpha$.}
\end{cases}
\end{equation*}

The homeomorphism for $\mathcal A$-region with a sink can be constructed similarly. Thus we have a homeomorphism 
\begin{equation*}
h_{\mathcal A}\colon cl(\mathcal A)\to cl(\mathcal A')
\end{equation*}
for every $\mathcal A$-region of the flow $\phi^t$.

{\bf IV. $\mathcal L$-region.} Here we will follow to alike construction in \cite{GurKur}. Let us consider some $\mathcal L$-region of the flow $\phi^t$ with an unstable (for definiteness) limit cycle $\mathfrak c$ inside. Consider a region 
\begin{equation*}
(\pi^*_{\phi'^t})^{-1}\xi\pi^*_{\phi^t}(\mathcal L)
\end{equation*} 
of the flow $\phi'^t$. We perfectly know that it is an $\mathcal L'$-region of the flow $\phi'^t$ with an unstable limit cycle $\mathfrak c'$  inside of the same type as $\mathcal L$ because of directions of edges and their number. We also know that as limit cycles as cutting circles of $\mathcal L$ and $\mathcal L'$ are oriented consistently because of equal orientation of $\psi(\tau_{\mathcal L,\mathcal M})$ and $\tau_{\xi(\mathcal L),\xi(\mathcal M)}$.

{\bf 1.} Consider the case of the annulus.

{\bf Step 1.} Let $L^*$ and $L^{**}$ be the two connecting components of $\partial\mathcal L$ and let $L'^*=(\pi^*_{\phi'^t})^{-1}\xi\pi^*_{\phi^t}(L^*)$, $L'^{**}=(\pi^*_{\phi'^t})^{-1}\xi\pi^*_{\phi^t}(L^{**})$. Let $h^*\colon L^*\to L'^*$ and $h^{**}\colon L^{**}\to L'^{**}$ be the contractions of the homeomorphisms constructed before on the closures of the elementary regions adjoined to $\mathcal L$ ($\mathcal L'$) with $L^{*}$ and $L^{**}$ as their common boundary accordingly.

{\bf Step 2.} Recall that $\Sigma_p$ ($\Sigma_{p'}$) is the Poincar\'e's cross-section of $\mathfrak c$ ($\mathfrak c'$), $F_{p}$ ($F_{p'}$) is the Poincar\'e's map and $\{p\}=\Sigma_p\cap\mathfrak c$ ($\{p'\}=\Sigma_{p'}\cap\mathfrak c'$). By Proposition \ref{LokSoprDif} $F_p\in Diff^1(\Sigma_p)$. The point $p$ is a source of $F_p$. Let $m_{a,b},\,a,b\in\Sigma_p$ ($m_{a',b'},\,a',b'\in\Sigma_{p'}$) be the $\Sigma_p$'s segment restricted by the points $a$ and $b$ ($\Sigma_{p'}$'s segment restricted by the points $a'$ and $b'$) and $\mu_{a,b}$ ($\mu_{a',b'}$) be its length.

Let $\{x^*\}=\Sigma_p\cap L^*$ and $\{x^{**}\}=\Sigma_p\cap L^{**}$. Let $x'^*\in L'^*$ and $x'^{**}\in L'^*$ be such that $x'^*=h^*(x^*)$ and $x'^{**}=h^{**}(x^{**})$. Let $\{x'_*\}=\Sigma_{p'}\cap L'^*$ and $\{x'_{**}\}=\Sigma_{p'}\cap L'^{**}$. Let $t^*\geq 0$ and $t^{**}\geq 0$ be the least non negative numbers such that $x'^*=\phi'^{t^*}(x'_*)$ and $x'^{**}=\phi'^{t^{**}}(x'_{**})$. Let 
\begin{align*}
p'^*=&\;\;\phi'^{(\frac{\mu_{x'_{**},p'}}{\mu_{x'_{**},x'_*}}(t^*-t^{**})+t^{**})}(p'), \\ \Sigma_{p'^*}=&\left\{\phi'^{(\frac{\mu_{x'_{**},x'}}{\mu_{x'_{**},x'_*}}(t^*-t^{**})+t^{**})}(x'):\;\; x'\in\Sigma_{p'}\right\}.
\end{align*} 
 
{\bf Step 3.} Let us construct a homeomorphism $h_\Sigma\colon\Sigma_p\to\Sigma_{p'^*}$ by the next way. For  $x\in m_{x^*,F^{-1}_p(x^*)}$ let $t^*_x\geq 0$ be such that $\phi^{t^*_x}(x)\in L^*$. Let $t^*_{x'}\geq 0$ be such that $\phi'^{(-t^*_{x'})}(h^*(\phi^{t^*_x}(x)))\in m_{x'^*,F^{-1}_{p'^*}(x'^*)}$. Then 
\begin{equation*}
h_\Sigma(x)=\phi'^{(-t^*_{x'})}(h^*(\phi^{t^*_x}(x)));
\end{equation*} 
Similarly for $x\in m_{x^{**},F^{-1}_p(x^{**})}$ let $t^{**}_x\geq 0$ be such that $\phi^{t^{**}_x}(x)\in L^{**}$. Let $t^{**}_{x'}\geq 0$ be such that $\phi'^{(-t^{**}_{x'})}(h^{**}(\phi^{t^{**}_x}(x)))\in m_{x'^{**},F^{-1}_{p'^*}(x'^{**})}.$ Then 
\begin{equation*}
h_\Sigma(x)=\phi'^{(-t^{**}_{x'})}(h^{**}(\phi^{t^{**}_x}(x)));
\end{equation*}
For $x\in m_{F^{-k}_p(x^*),F^{-k+1}_p(x^*)}$, where $k\in\mathbb N$ let 
\begin{equation*}
h_\Sigma(x)=F_{p'^*}^{-k}(x)\circ h_\Sigma\circ F_p^{k}(x);
\end{equation*}
For $x\in m_{F^{-l}_p(x^{**}),F^{-l+1}_p(x^{**})}$, where $l\in\mathbb N$ let 
\begin{equation*}
h_\Sigma(x)=F_{p'^*}^{-l}(x)\circ h_\Sigma\circ F_p^{l}(x).
\end{equation*}

{\bf Step 4.} Let us define a homeomorphism $h_{\mathcal L}\colon cl(\mathcal L) \to cl(\mathcal L')$ by the next formulas.
\begin{align*}
&\text{For }x\in\Sigma_p\setminus(m_{F_p^{-1}(x^*),x^*}\cup m_{F_p^{-1}(x^{**}),x^{**}}) \\ &\text{let }h_\mathcal L|_{l_{x,F_p(x)}}=h_{l_{x,F_p(x)}}\colon l_{x,F_p(x)}\to l_{h_\Sigma(x),h_\Sigma(F_p(x))}.
\\ &\text{For }x\in m_{F_p^{-1}(x^*),x^*} \\ &\text{let }h_\mathcal L|_{l_{x,\phi^{t^*_x}(x)}}=h_{l_{x,\phi^{t^*_x}(x)}}\colon l_{x,\phi^{t^*_x}(x)}\to l_{h_\Sigma(x),h^*(\phi^{t^*_x}(x))}. \\
&\text{For }x\in m_{F_p^{-1}(x^{**}),x^{**}} \\ &\text{let }h_\mathcal L|_{l_{x,\phi^{t^{**}_x}(x)}}=h_{l_{x,\phi^{t^{**}_x}(x)}}\colon l_{x,\phi^{t^{**}_x}(x)}\to l_{h_\Sigma(x),h^{**}(\phi^{t^{**}_x}(x))}.
\end{align*}

{\bf 2.} Consider the case of the M\"obius band. In general the construction is similar to the case of the annulus but it has the few important differences.

{\bf Step 1.} The boundary $\partial\mathcal L$ has only one connected component, and $\Sigma_p$ crosses it in two points $x^*$ and $x^**$. Denote $h^*\colon\partial\mathcal L\to\partial\mathcal L$ the homeomorphism constructed before on $\partial \mathcal L$. Let $x'_*$ be one of the two points in which $\Sigma_{p'}$ crosses $\partial\mathcal L'$. Let $x'^*=h^*(x^*)$. Let $t^*\geq 0$ be the least non negative number such that $x'^*=\phi'^{t^*}(x'_*)$. Let 
\begin{align*}
p'^*=&\;\;\phi'^{t^*}(p') \quad \text{and} \quad \Sigma_{p'^*}=\left\{\phi'^{t^*}(x'):\;\; x'\in\Sigma_{p'}\right\}.
\end{align*} 
Denote by $x'^{**}$ the second point in which $\Sigma_{p'^*}$ crosses $\partial\mathcal L'$ (i.e. $x'^{**}\not=x'^*$).

{\bf Step 2.} Let us construct a homeomorphism 
\begin{equation*}
h_\Sigma\colon(\Sigma_p\setminus m_{x^{**},F^{-1}_p(x^*)})\to(\Sigma_{p'^*}\setminus m_{x'^{**},F^{-1}_{p'^*}(x'^*)})
\end{equation*} 
by the next way: 
For $x\in m_{x^*,F^{-2}_p(x^*)}$ let $t^*_x\geq 0$ be such that $\phi^{t^*_x}(x)\in\partial{\mathcal L}$. Let $t^*_{x'}\geq 0$ be such that $\phi'^{(-t^*_{x'})}(h^*(\phi^{t^*_x}(x)))\in m_{x'^*,F^{-2}_{p'^*}(x'^*)}$. Then 
\begin{equation*}
h_\Sigma(x)=\phi'^{(-t^*_{x'})}(h^*(\phi^{t^*_x}(x)));
\end{equation*}
For $x\in m_{F^{-k}_p(x^*),F^{-k-2}_p(x^*)}$, where $k\in\mathbb N$, let 
\begin{equation*}
h_\Sigma(x)=F_{p'^*}^{-k}(x)\circ h_\Sigma\circ F_p^k(x);
\end{equation*}

{\bf Step 3.} Let us define the homeomorphism $h_{\mathcal L}\colon cl(\mathcal L)\to cl(\mathcal L')$ by the next formulas  
\begin{align*}
&\text{For }x\in\Sigma_p\setminus(m_{F_p^{-2}(x^*),x^*}\cup m_{x^{**},F_p(x^*)}) \\ &\text{let }h_\mathcal L|_{l_{x,F_p(x)}}=h_{l_{x,F_p(x)}}\colon l_{x,F_p(x)}\to l_{h_\Sigma(x),h_\Sigma(F_p(x))}.\\
&\text{For } x\in m_{F_p^{-2}(x^*),x^*}\\ &\text{let } h_\mathcal L|_{l_{x,\phi^{t^*_x}(x)}}=h_{l_{x,\phi^{t^*_x}(x)}}\colon l_{x,\phi^{t^*_x}(x)}\to l_{h_\Sigma(x),h^*(\phi^{t^*_x}(x))}.
\end{align*}

The homeomorphism for $\mathcal L$-region with a stable limit cycle can be constructed similarly. Thus we have a homeomorphism 
\begin{equation*}
h_{\mathcal L}\colon cl(\mathcal L)\to cl(\mathcal L')
\end{equation*} 
for every $\mathcal L$-region of the flow $\phi^t$.

{\bf The final homeomorphism.} We have created the homeomorphism for each elementary region. Thus, the final homeomorphism $h\colon S\to S$ we define by the formula 
\begin{equation*}
h(x)=\begin{cases} 
h_{\mathcal A}(x) &\text{if $x\in cl(\mathcal A)$,}\\
h_{\mathcal E}(x) &\text{if $x\in cl(\mathcal E)$,}\\
h_{\mathcal M}(x) &\text{if $x\in cl(\mathcal M)$,}\\
h_{\mathcal L}(x) &\text{if $x\in cl(\mathcal L)$.}
\end{cases}
\end{equation*}

So, Theorem \ref{Classif} is proved.

\begin{flushright}$\square$\end{flushright}

\section{Realisation of an admissible equipped graph $\Upsilon^*$ by the $\Omega$-stable flow $\phi^t$ on a surface $S$}

Firstly we give the proof of the Lemma \ref{ReBezPer} about realisation of an admissible four-colour graph by the $\Omega$-stable flow $f^t$ without limit cycles.

\subsection{The proof of the Lemma \ref{ReBezPer}}

This proof is equal to the one in our paper \cite{KruMaPoMS} but still we give it there for completness.

Let $\Gamma$ be some admissible four-colour graph.

{\bf I.} Let us construct an $\Omega$-stable flow  $f^t$ without limit cycles corresponding to $\Gamma$'s isomorphic class step by step. 

{\it Step 1.} Consider some vertex $b$ of $\Gamma$. The vertex $b$ is incident to $n$ edges, first of which is a $t$-edge, second one is an $u$-edge, third one is a $s$-edge and rest ones are $c^b_j$-edges, $j=\overline{1,(n-3)}$. We construct on $\mathbb R^2$ a regular $2(n-1)$-gon $A_1 A_2\dots A_{2(n-1)}$ with the centre in the origin $O(0,0)$ and the vertices $A_1 (1,0)$ and $A_n (-1,0)$ (see Fig. \ref{OblastNaPl}). Denote by $\varphi$ the central angle and by $a$ the length of a side of $A_1 A_2\dots A_{2(n-1)}$. Then $$\varphi=\frac{\pi}{n-1},\,a=\frac{1}{\sin\varphi}.$$ Hence, $A_k=(\cos(k-1)\varphi,\sin(k-1)\varphi)$ for $k=\overline{1,2(n-1)}$.

\begin{figure}[htb]
\centerline{\includegraphics [width=13.5 cm] {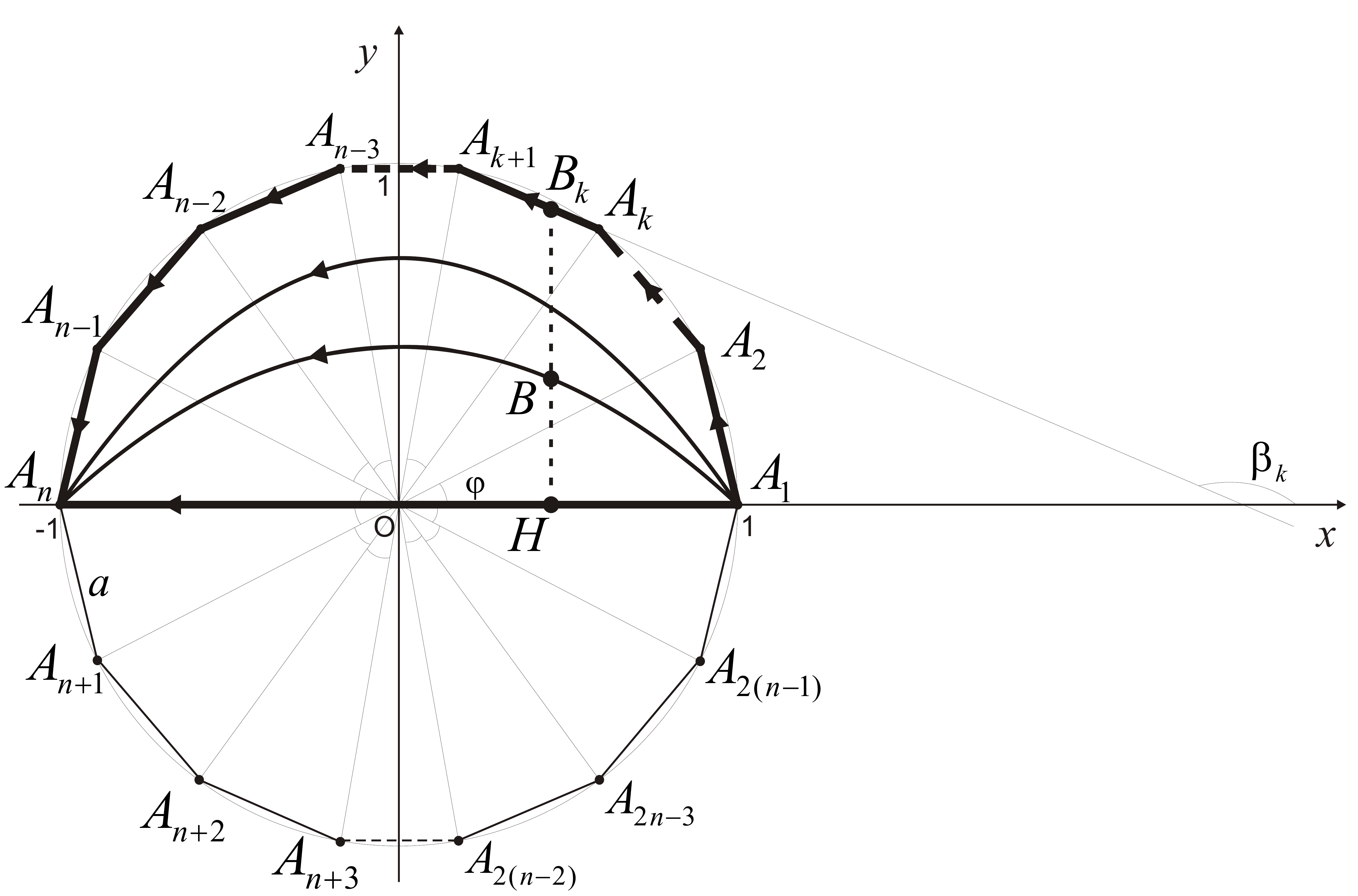}} \caption{Designing of the vector field $v_b$}\label{OblastNaPl}
\end{figure}

Let us denote $M_b\equiv cl(A_1 A_2\dots A_{2(n-1)}\cap\{(x,y)\in\mathbb R^2\mid y>0\})$. By construction $M_b$ is the $n$-gon with the vertices $A_1,A_2,\dots,A_n$, i.e. the number of $M_b$'s vertices is equal to the number of the edges incident to $b$. We will call $A_1 A_n$ as the $t$-side, $A_{n-1},A_n$ as the $u$-side or the $c_0$-side, $A_1,A_2$ as the $s$-side or the $c_{n-2}$-side and $A_k A_{k+1}$ as the $c_{n-k-1}$-side, where $k=\overline{2,(n-2)}$.

{\it Step 2.} Let us design the vector field $v_b$ on $M_b$ the following way.

Firstly we define the vector field $v_{A_1,A_n}$ on the side $A_1 A_n$ by the differential equations system 
\begin{equation*}
\left\{
\begin{aligned} 
&\dot y=0, \\ 
& \dot x=\sin\frac{1}{2}\pi(x-1).
\end{aligned}
\right. 
\end{equation*} 
By construction $A_1$ and $A_n$ are fixed points, and the flow given by $v_{A_1,A_n}$ moves from $A_1$ to $A_n$. Let us define the vector field on the other sides of $M_b$. 

Consider the side $A_k A_{k+1},\,k=\overline{1,(n-1)}$. The straight line passing through the points $A_k$, $A_{k+1}$ is defined by the equation 
\begin{equation*}
A_k A_{k+1}\colon\frac{x-\cos(k-1)\varphi}{\cos k\varphi-\cos(k-1)\varphi}=\frac{y-\sin(k-1)\varphi}{\sin k\varphi-\sin(k-1)\varphi},
\end{equation*} 
it gives us its slope $\beta_k$: 
\begin{equation*}
\beta_k=\arctan\frac{\sin k\varphi-\sin(k-1)\varphi}{\cos k\varphi-\cos(k-1)\varphi}.
\end{equation*} 
Now we reduce the considered case to the case of $A_1 A_n$. To do this let us make the one-to-one correspondence $t_k$ between points of $[\cos k\varphi,\cos(k-1)\varphi]$ and $[-1,1]$ by the formula 
\begin{equation*}
t_k=2\frac{x-\cos k\varphi}{\cos(k-1)\varphi-\cos k\varphi}-1.
\end{equation*} 
Let $\gamma_k=\sin\frac{1}{2}\pi(t_k-1)$, then we define the vector field $v_{A_k A_{k+1}}$ by the following system of equations 
\begin{equation*}
\left[
\begin{aligned}
& \left\{
\begin{aligned} 
& \beta_k\not=0, \\
& \dot x=-\gamma_k\cdot \cos\beta_k\cdot \text{sign}\,x, \\ 
& \dot y=-\gamma_k\cdot \sin\beta_k\cdot \text{sign}\,x, \end{aligned}
\right. \\
& \left\{
\begin{aligned}
 & \beta_k=0, \\
& \dot x=\gamma_k, \\ & \dot y=0. 
\end{aligned}
\right. \\
\end{aligned}
\right. 
\end{equation*}

{\it Step 3.} Now we define the vector field $v_{int}$ inside $M_b$. Let us take an arbitrary point $B(x,y)\in int M_b$. Then $B\in B_k H$, where $B_k\in A_k A_{k+1}$ for some $k=\overline{1,(k-1)}$ and $H$ is the $B_k$'s projection to $A_1 A_n$ (see Fig. \ref{OblastNaPl}). Define $v_{int}$ as an average between the vectors $v_{A_1,A_n}(H)$ and $v_{A_k A_{k+1}}(B_k)$ by the following system of equations 
\begin{equation*}
\left[
\begin{aligned} 
& \left\{
\begin{aligned} 
& \beta_k\not=0, \\
& \dot x=\frac{B_k B}{B_k H}\sin\frac{1}{2}\pi(x-1)-\frac{BH}{B_k H}\gamma_k\cdot \cos\beta_k\cdot \text{sign}\,x, \\ 
& \dot y=-\frac{BH}{B_k H}\gamma_k\cdot \sin\beta_k\cdot \text{sign}\,x, \end{aligned}
\right. \\
& \left\{
\begin{aligned} 
& \beta_k=0, \\
& \dot x=\frac{B_k B}{B_k H}\sin\frac{1}{2}\pi(x-1)+\frac{BH}{B_k H}\gamma_k, \\ 
& \dot y=0. 
\end{aligned}
\right. \\
\end{aligned}
\right. 
\end{equation*}

Define $v_b$ by the system 
\begin{equation*}
v(x,y)=\begin{cases}
v_{A_1 A_n}(x,y), & (x,y)\in A_1 A_n, \\
v_{A_k A_{k+1}}(x,y), & (x,y)\in A_k A_{k+1}, k\in\{1,\dots,n-1\}, \\
v_{int}(x,y), & (x,y)\in int\,M_b. 
\end{cases} 
\end{equation*}

{\it Step 4.} We denote by $B$ the set of vertices, by $N$ -- the number of vertices, by $E$ -- the set of edges of $\Gamma$. Let $\eta_b$ is the correspondence between $t$-, $u$-, $s$- or $c_i$-edge incident to the vertex $b$ and $t$-, $u$-, $s$- or $c_i$-side of $M_b$ accordingly. Denote by $\mathfrak M$ the disjunctive union of $M_b$, $b\in B$. Introduce on $\mathfrak M$ the minimal equivalence relation satisfying to the following rule: if $b_1,\,b_2\in B$ are incident to $e\in E$, then the segments $P_1 Q_1=\eta_{b_1}(e)$ and $P_2 Q_2=\eta_{b_2}(e)$ are identified so that a point $(x_1,y_1)\in P_1 Q_1=[(x_{P_1},y_{P_1}),(x_{Q_1},y_{Q_1})]$ is equivalent to the point $(x_2,y_2)\in P_2 Q_2=[(x_{P_2},y_{P_2}),(x_{Q_2},y_{Q_2})]$, where 
\begin{equation*}
\left\{
\begin{aligned}
& x_2=x_{P_2}+\frac{(x_1-x_{P_1})(x_{Q_2}-x_{P_2})}{x_{Q_1}-x_{P_1}}, \\
& y_2=y_{P_2}+\frac{(y_1-y_{P_1})(y_{Q_2}-y_{P_2})}{y_{Q_1}-y_{P_1}}. 
\end{aligned}
\right.
\end{equation*} 

Properties of an admissible graph entail that the quotient space $M=\mathfrak M/\sim$ is an closed topological 2-manifold. Denote by $q\colon\mathfrak M\to S$ its natural projection. Notice that the vector field in the points equivalent by $\sim$ has equal length, hence, $q$ induces the continuous vector field, we denote it by $V_M$. 

{\it Step 5.} Let us define a smooth structure on $M$ such that $V_M$ is smooth on it.

Let us cover $M$ by a finite number of maps $(U_z,\psi_z),\,z\in M$, where $U_z\subset M$ is the open neighbourhood of $z$ and $\psi_z\colon U_z\to\mathbb R^2$ is the homeomorphism to the image of the following type.

\begin{figure}[h!]
\centerline{\includegraphics [width=15 cm] {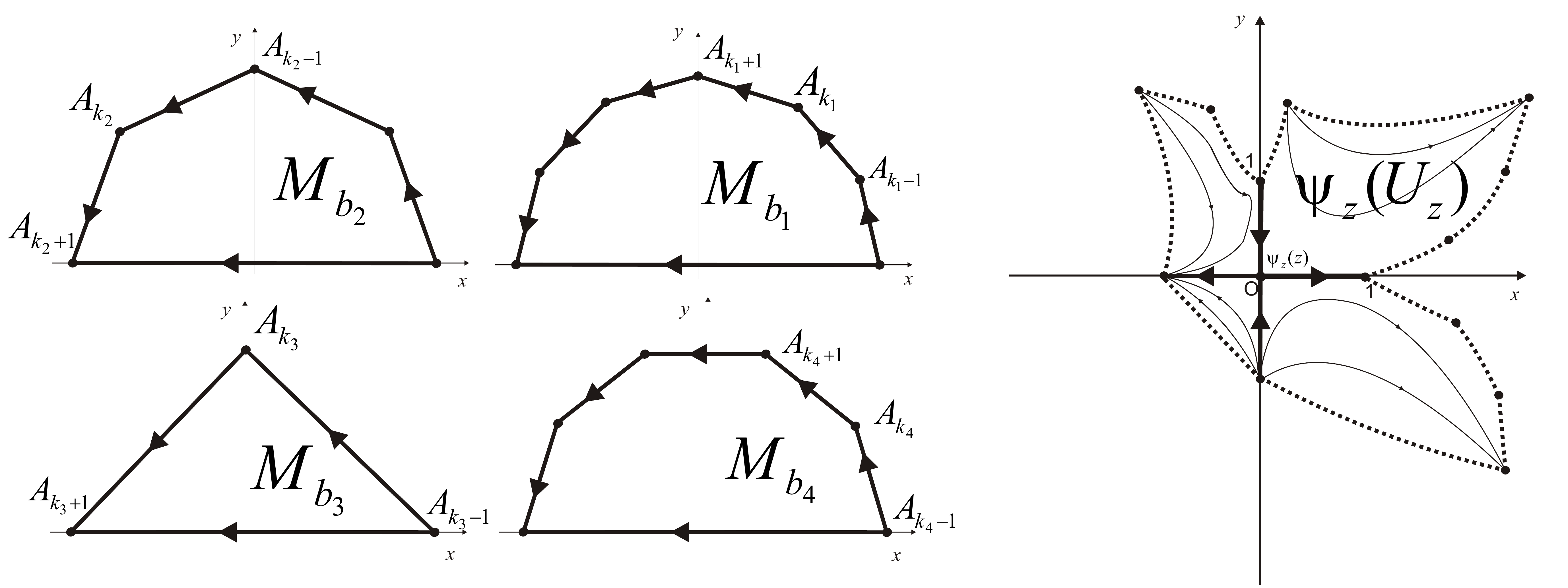}} \caption{An example of a map of the first type}\label{SglazhSedlo}
\end{figure}

{\bf 1.} Consider on $\Gamma$ a $c^*$-cycle 
\begin{equation*}
\{b_1, c^{b_1}_{j_1}\!=\!c^{b_2}_{j_2}, b_2, c^{b_2}_{j_2-1}\!=\!c^{b_3}_{j_3-1}, b_3, c^{b_3}_{j_3}\!=\!c^{b_4}_{j_4}, b_4, c^{b_4}_{j_4-1}\!=\!c^{b_1}_{j_1-1}, b_1\},
\end{equation*} 
where a vertix $b_i\in B,\,i\in\{1,\dots,4\}$ corresponds to the $n_i$-gon $M_{b_i}$ and $\eta_{b_i}(c^{b_i}_{j_i})=A_{k_i}A_{k_{i+1}}$ for $k_i=n_i-j_i-1$ (see Fig. \ref{SglazhSedlo}). We will denote the length of $A_{k_i} A_{k_{i+1}}$, the central angle of $M_{b_i}$ and the angle between the vector $\overrightarrow{A_{k_i} A_{k_i+1}}\,(\overrightarrow{A_{k_i} A_{k_i-1}})$ and $Ox^+$ by $a_i$, $\varphi_i$ and $\beta^+_{k_i}\,(\beta^-_{k_i})$ accordingly. Besides, the angles $\beta^+_{k_i},\,\beta^-_{k_i}$ are selected so that $|\beta^+_{k_i}-\beta^-_{k_i}|<\pi$. Let \begin{multline*}
U_z=int(\bigcup\limits_{i=1}^4 q(M_{b_i})),\; \psi_z(\varrho)=\mu_i(p_{1,i}((q|_{M_{b_i}})^{-1}(\varrho))) \\ \textrm{ for }\varrho\in q(M_{b_i}), i\in\{1,\dots,4\},
\end{multline*} 
where 
\begin{equation*}
p_{1,i}(x,y)=\left(\frac{x-\cos(k_i-1)\varphi_i}{a_i}, \frac{y-\cos(k_i-1)\varphi_i}{a_i}\right)
\end{equation*} 
and 
\begin{equation*}
\mu_i(x,y)=\mu_i(r\cos\theta,r\sin\theta)=(r\cos\theta_{1,i},r\sin\theta_{1,i}),
\end{equation*} 
where $(r,\theta)$ are polar coordinates and the function $\theta_{1,i}(\theta)$ is given by the formula 
\begin{equation*}
\theta_{1,i}(\theta)=\left(i-2\cdot\left(\frac{i}{2}(mod1)\right)\right)\cdot\frac{\pi}{2}+(-1)^{i-1}\cdot\frac{\pi}{2}\cdot\frac{\theta-\beta^+_{k_i}}{\beta^+_{k_i}-\beta^-_{k_i}}.
\end{equation*}

Here the function $p_{1,i}(x,y)$ produces parallel transfer of $M_{b_i}$ so that the vertex $A_{k_i}$ hits in the origin, and increases the lengths of $A_{k_i} A_{{k_i}+1}$ and $A_{k_i-1} A_{k_i}$ up to unit. The function $\mu_i(x,y)$ identifies the angle of the vertex $A_{k_i}$ with $i$-th coordinate angle.

{\bf 2.} Consider on $\Gamma$ a $st$-cycle 
\begin{equation*}
\{b_1, (b_1,b_2), b_2, (b_2, b_3), b_3,\dots, b_{2m-1}, (b_{2m-1},b_{2m}), b_{2m}, (b_{2m}, b_1), b_1\},
\end{equation*} 
where a vertex $b_i\in B$, $i\in\{1,\dots,2m\}$ corresponds to the $n_i$-gon $M_{b_i}$, $\eta_{b_{2j-1}}((b_{2j-1},b_{2j}))$ is the side $A_1 A_2$ of $M_{b_{2j-1}}$,\\ $\eta_{b_{2j}}((b_{2j-1},b_{2j}))$ is the side $A_1 A_2$ of $M_{b_{2j}}$,\\ $\eta_{b_{2j}}((b_{2j}, b_{2j+1}))$ is the side $A_1 A_{n_{2j}}$ of $M_{b_{2j}}$,\\ $\eta_{b_{2j+1}}((b_{2j}, b_{2j+1}))$ is the side $A_1 A_{n_{2j+1}}$ of $M_{b_{2j+1}}$ for $j\in\{1,\dots,m\}$, $n_{2j+1}=n_1$. 

Recall that the length of $A_1 A_2$ of $M_{b_i}$ is equal to $a_i$ and the length of $A_1 A_{n_i}$ is equal to 2. Denote the angle between the vector $\overrightarrow{A_1 A_2}$ and $Ox^+$ by $\beta_{1,i}^+$. Let 
\begin{multline*}
U_z=int(\bigcup\limits_{i=1}^{2m}q(\mathcal M_{b_i})),\; \psi_z(\varrho)=\nu_i(p_{2,i}(q|_{ M_{b_i}})^{-1}(\varrho)))\\ \textrm{ for }\varrho\in q(M_{b_i}), i\in\{1,\dots,2m\},
\end{multline*} 
where 
\begin{equation*}
p_{2,i}(x,y)=(x-1,y)
\end{equation*} 
and 
\begin{equation*}
\nu_i(x,y)=\nu_i(r cos\theta,r sin\theta)=(r_{2,i}(r,\theta)\cdot cos(\theta_{2,i}(\theta)),\; r_{2,i}(r,\theta)\cdot sin(\theta_{2,i}(\theta)))
\end{equation*}  
are given by the formulas 
\begin{align*}
& r_{2,i}(r,\theta)=\frac{r}{2}\cdot\frac{\theta-\beta^+_{1,i}}{\pi-\beta^+_{1,i}}+\frac{r}{a_i}\cdot\frac{\pi-\theta}{\pi-\beta^+_{1,i}},\\ 
& \theta_{2,i}(\theta)=\left(i-2\cdot\left(\frac{i}{2}(mod1)\right)\right)\cdot\frac{\pi}{m}+(-1)^{i-1}\cdot\frac{\theta-\beta^+_{1,i}}{\pi-\beta^+_{1,i}}\cdot\frac{\pi}{m}.
\end{align*}

Here the function $p_{2,i}(x,y)$ produces parallel transfer of $M_{b_i}$ so that the vertex $A_1$ hits in the origin. The function $\nu_i(x,y)$, $i\in\{1,\dots,2m\}$ changes the lengths of $A_1 A_2$ and $A_1 A_{n_i}$ to unit, changes the quantity of the angle of the vertex $A_1$ to $\frac{\pi}{m}$ and distributes the polygons $M_{b_i}$ with their $A_1$ to the origin so that the angles of $A_1$ adjoin each other and fill the full angle distributing each on $i$-th place while by-passing around the origin from $Ox^+$ counter clockwise on some circle with a radius less than 1. Also they provide a coincidence of the same-colour sides of adjoining polygons.

{\bf 3.} Consider on $\Gamma$ a $ut$-cycle 
\begin{equation*}
\{b_1, (b_1,b_2), b_2, (b_2, b_3), b_3,\dots, b_{2m-1}, (b_{2m-1},b_{2m}), b_{2m}, (b_{2m}, b_1), b_1\},
\end{equation*} 
where a vertex $b_i\in B$, $i\in\{1,\dots,2m\}$ corresponds to the $n_i$-gon $M_{b_i}$, $\eta_{b_{2j-1}}((b_{2j-1},b_{2j}))$ is the side $A_{n_{2j-1}-1}A_{n_{2j-1}}$ of $M_{b_{2j-1}}$,\\ $\eta_{b_{2j}}((b_{2j-1},b_{2j}))$ is the side $A_{n_{2j}-1}A_{n_{2j}}$ of $M_{b_{2j}}$,\\ $\eta_{b_{2j}}((b_{2j}, b_{2j+1}))$ is the side $A_1 A_{n_{2j}}$ of $M_{b_{2j}}$,\\ $\eta_{b_{2j+1}}((b_{2j}, b_{2j+1}))$ is the side $A_1 A_{n_{2j+1}}$ of $M_{b_{2j+1}}$ for $j\in\{1,\dots,m\}$, $n_{2j+1}=n_1$.

Recall that the length of $A_{n_i-1} A_{n_i}$ of $M_{b_i}$ is equal to $a_i$, the length of $A_1 A_{n_i}$ is equal to 2, the angle between the vector $\overrightarrow{A_{n_i} A_{n_i-1}}$ and $Ox^+$ is equal to $\beta_{n_i,i}^-$. Let 
\begin{multline*}
U_z=int(\bigcup\limits_{i=1}^{2m} q( M_{b_i})),\;\psi_z(\varrho)=\kappa_i(p_{3,i}((q|_{M_{b_i}})^{-1}(\varrho)))\\ \textrm{ for }\varrho\in q(M_{b_i}), i=\{1,\dots,2m\},
\end{multline*} 
where $$p_{3,i}(x,y)=(x+1,y)$$ and 
\begin{equation*}
\kappa_i(x,y)=\kappa_i(r cos\theta,r sin\theta)=(r_{3,i}(r,\theta)\cdot cos(\theta_{3,i}(\theta)),\; r_{3,i}(r,\theta)\cdot sin(\theta_{3,i}(\theta)))
\end{equation*}
are given by the formulas 
\begin{align*}
& r_{3,i}(r,\theta)=\frac{r}{2}\cdot\frac{\beta^-_{n_i}-\theta}{\beta^-_{n_i}}+\frac{r}{a_i}\cdot\frac{\theta}{\beta^-_{n_i}},\\
& \theta_{3,i}(\theta)=\left(i-2\left(\frac{i}{2}(mod 1)\right)\right)\cdot\frac{\pi}{m}+(-1)^{i-1}\frac{\theta}{\beta^-_{n_i}}\cdot\frac{\pi}{m}.
\end{align*}

Here the function $p_{3,i}(x,y)$ produces parallel transfer of $M_{b_i}$ so that the vertex $A_{n_i}$ hits in the origin. The function $\kappa_i(x,y)$, $i\in\{1,\dots,2m\}$ changes the lengths of $A_{n_i-1} A_{n_i}$ and $A_1 A_{n_i}$ to unit preserving continuity of the field, changes the quantity of the angle of the vertex $A_{n_i}$ to $\frac{\pi}{m}$ and distributes the polygons $M_{b_i}$ with their $A_{n_i}$ to the origin so that the angles of $A_{n_i}$ adjoin each other and fill the full angle distributing each on $i$-th place while by-passing around the origin from $Ox^+$ counter clockwise on some circle with a radius less than 1. Also they provide coincidence of same-colour sides of adjoining polygons.

The conversion functions for introduced maps are the compositions of smooth maps constructed in 1-3 and the inverse ones for them, hence, these maps design a smooth structure on the surface $M$.

{\bf II.} Here we prove i) and ii) of Theorem \ref{ReBezPer}.

i) Let us prove that the Euler characteristic of $M$ may be found by the formula (\ref{EjHarBezPer}) $\chi(S)=\nu_0-\nu_1+\nu_2$, where $\nu_0$, $\nu_1$, and $\nu_2$ is the numbers of all $tu$-, $c^*$- and $st$-cycles of $\Gamma$ accordingly. The fact that the numbers of all the sinks, the saddle points and the sources are equal to $\nu_0$, $\nu_1$ and $\nu_2$  accordingly follows from Proposition \ref{TochkiICikly}. That entails the affirmation i), because the given formula is the formula for the index sum of the singular points of $f^t$.

{\bf III.} Let us prove that the surface $M$ is non-orientable if and only if $\Gamma$ contains at least one cycle of odd length.

The surface $M$ with the flow $f^t$ is orientable if and only if all polygonal regions of $f^t$ can be oriented consistently. We can define an orientation for each polygonal region by selection of one of two possible cyclic order of its fixed points: $\alpha$-$\sigma_n$-$\dots$-$\sigma_1$-$\omega$ and $\omega$-$\sigma_1$-$\dots$-$\sigma_n$-$\alpha$, where $\alpha$ is a source, $\sigma_j$ is a saddle point ($j=\overline{1,n}$), $\omega$ is a sink. Let the sign ``$+$'' is appropriated to a polygonal region in the first case, ``$-$'' -- in the second one. It is clear that orientations of two such regions can are consistent if and only if the regions are equipped by different signs. As there is one-to-one correspondence $\pi_{f^t}$ between the polygonal regions of $f^t$ and the vertices of the graph $\Gamma$, then the condition of orientability of $M$ may be stated the following way: the surface $M$ can be oriented if and only if the vertices of $\Gamma$ are equipped by the signs ``$+$'' and ``$-$'' so that each two its vertices connected by an edge has different signs. We call such arrangement of signs of the $\Gamma$'s vertices the {\it right} one.

So all we need is to prove that $\Gamma$ doesn't have odd length cycles if and only if the right sign arrangement for the vertices of $\Gamma$ exists.

Truth of that affirmation from the left to the right is obvious, because the right sign arrangement in an odd length cycle is impossible. Let us prove from the right to the left: let $\Gamma$ doesn't have odd length cycles. Then the right sign arrangement might be made this way: let us take some vertex $b_0$ of $\Gamma$ and appropriate to it ``$+$''; for each other vertex $b_i$ let us consider a path connecting $b_i$ with $b_0$ and appropriate to it ``$+$'' if the path has even length and ``$-$'' in the other case. As we suppose $\Gamma$ doesn't have odd length cycles, then such arrangement doesn't depend on the selection of a path and, hence, is defined correctly.

\begin{flushright}$\square$\end{flushright}

\subsection{The proof of the realisation Theorem \ref{Realiz}}

Let $\Upsilon^*$ be some admissible equipped graph.

{\bf I.} Let us construct an $\Omega$-stable flow $\phi^t$ corresponding to $\Upsilon^*$'s isomorphic class by creation the surface $S$ and the continuous vector field.

{\bf Step 1.} Let $B$ be the set of $\Upsilon^*$'s vertices and $E$ be the set of its edges. Let us construct for every $b\in B$ a surface $S_b$ with a boundary and a vector field $\overrightarrow{V_b}$ on it, transversal to the boundary. The required  $\Omega$-stable flow on $S$ will be glued from these pieces of dynamics by means annuli which correspond to the edges from $E$ according to incidence. 

{\bf $\mathcal A$-vertex.}  Let $b$ be an $\mathcal A$-vertex. Then $S_b=\{(x,y)\in\mathbb R^2\mid x^2+y^2<1\}$ and the vector field on the disk $S_b$ we define by the vector-function $\overrightarrow{V_b}(x,y)=\{-x,-y\}$ ($\overrightarrow{V_b}(x,y)=\{x,y\}$), if the edges incident to $b$ are directed to $b$ (out of $b$). 

{\bf $\mathcal E$-vertex.} Let $b$ be an $\mathcal E$-vertex. Let $W=[0,1]\times[0,1]$. Define the minimal equivalence relation $\sim_{\mathcal E}$ on $W$ such that $(x,0)\sim_{\mathcal E}(x,1)$ for $x\in[0,1]$. Let $S_b=W/\sim_{\mathcal E}$ and $q_{b}\colon W\to S_b$ be the natural projection. Define on the annulus $S_b$ the vector field by the formula $\overrightarrow{V_b}(x,y)=q_b(\{\frac{1}{2},1\})$ ($\overrightarrow{V_b}(x,y)=q_b(\{\sin\frac{2\pi}{3}\big(x+\frac{1}{4}\big),\cos\frac{2\pi}{3}\big(x+\frac{1}{4}\big)\})$), if the weight of $\mathcal E$ is ``$+$'' (``$-$''). 

{\bf $\mathcal L$-vertex.} Let $b$ be an $\mathcal L$-vertex. Let 
\begin{equation*}
\mathcal W=\{(x,y)\in\mathbb R^2: |x|\leq \frac{3-\cos\pi y}{2},\;0\leq y\leq 1\}.
\end{equation*} 
Then $\mathcal W$ is a curvilinear trapezium with the vertices $A(-1;0),B(-2;1),C(2;1),D(1;0)$. Define on $\mathcal W$  the minimal equivalence relation $\sim_{\mathcal L}$ such that $(x,0)\sim_{\mathcal L}(2x,1)$ ($(x,0)\sim_{\mathcal L}(-2x,1)$) for $x\in AD$, if the vertex $b$ is incident to two edges (one edge). Let $S_b=\mathcal W/\sim_{\mathcal L}$ and let $q_{b}\colon\mathcal W\to S_b$ be its natural projection. Then $S_b$ is the annulus (the M\"obius band). Define on $S_b$ the vector field by the formula $\overrightarrow{V_b}(x,y)=q_b(\{0,1\})$ ($\overrightarrow{V_b}(x,y)=q_b(\{0,-1\})$) and orient the boundary of $S_b$ in the direction of motion along the coordinate $y$ from $0$ to $1$ (from $1$ to $0$), if the edges incident to $b$ are directed to $b$ (out of $b$). 

{\bf $\mathcal M$-vertex.} Let $b$ be a $\mathcal M$-vertex. Then $b$ is equipped with the four colour graph $\Gamma_{\mathcal M}$, corresponding to the surface $M$ with the vector field $\overrightarrow{V_M}$, constructed in the proof of Lemma \ref{ReBezPer}. Let $\omega$ ($\alpha$) be a sink (source) of $\overrightarrow{V_M}$ such that $\omega=\pi_{V_M}^{-1}(\tau_{b,\mathcal L})$ ($\alpha=\pi_{V_M}^{-1}(\tau_{\mathcal L,b})$), where $\pi_{V_M}$ is the one-to-one correspondence between the elements of the field $\overrightarrow{V_M}$ and the elements of the four colour graph $\Gamma_\mathcal M$. Let $u_\omega$ ($u_\alpha$) is some neighbourhood of $\omega$ (of $\alpha$) without other elements of the basic set inside and with the boundary transversal to the trajectories of $\overrightarrow{V_M}$. Let us orient $\partial u_\omega$ ($\partial u_\alpha$) consistently with the orientation of the cycle $\tau_{b,\mathcal L}$ ($\tau_{\mathcal L,b}$). Then $S_b=M\setminus\bigcup\limits_{\omega=\pi_{V_M}^{-1}(\tau_{b,\mathcal L})}int~u_\omega\cup\bigcup\limits_{\alpha=\pi_{V_M}^{-1}(\tau_{\mathcal L},b)}int~u_\alpha$ with the field $\overrightarrow{V_b}=\overrightarrow{V_M}|_{S_b}$. We will suppose that each connected component of $\partial S_b$ has an orientation due to the oriented cycle the orientation.

{\bf Step 2}. Let $A=\mathbb S^1\times[-1,1]$ and we have two vector fields $\overrightarrow{V^-}=\{\overrightarrow{v^-}(s),s\in\mathbb S^1\}$, $\overrightarrow{V^+}=\{\overrightarrow{v^+}(s),s\in\mathbb S^1\}$ on $\mathbb S^1\times\{-1\}$, $\mathbb S^1\times\{1\}$, accordingly, such that they are transversal to $\partial{A}$, $\overrightarrow{V^-}$ has a direction to $A$, $\overrightarrow{V^+}$ has a direction out of $A$. Let  
\begin{equation*}
\overrightarrow{V_{A}}=\{\overrightarrow{v}(s,t)=\frac12\left((1-t)\overrightarrow{v^-}(s)+(1+t)\overrightarrow{v^+}(s)\right),s\in\mathbb S^1,t\in[-1,1]\}.
\end{equation*}
We will called the vector field $\overrightarrow{V_{A}}$ by {\it an average of the boundaries}.

For every edge $e\in E$ denote by $A_e$ a copy of the annulus $A$. Let us notice that the sets $\partial\left(\bigsqcup\limits_{b\in B} S_{b}\right)$ and $\partial\left(\bigsqcup\limits_{e\in E}A_e\right)$ consist of the same number of circles. Let $h_{\Upsilon^*}:\partial\left(\bigsqcup\limits_{b\in B} S_{b}\right)\to\partial\left(\bigsqcup\limits_{e\in E}A_e\right)$ be a diffeomorphism such that if $h_{\Upsilon^*}(x)=y$ for $x\in S_b$, $y\in A_e$ then $b,e$ are incident, moreover, $h_{\Upsilon^*}$ induces a concordant orientation on the connected components of $\partial A_e$ for the edge $e$ which is incident to  $\mathcal M$-vertex and $\mathcal L$-vertex.    

Let $\mathcal S=\bigsqcup\limits_{b\in B} S_{b}\sqcup\bigsqcup\limits_{e\in E}\mathcal A_e$. Let us introduce on $\mathcal S$ the minimal equivalent relation $\sim_{\Upsilon^*}$ such that $x\sim_{\Upsilon^*}h_{\Upsilon^*}(x)$. Then $\mathcal S/\sim_{\Upsilon^*}$ is a closed surface, denote it by $S$ and  by $q_S\colon\mathcal S\to S$ the natural projection. Then the required vector field $\overrightarrow{V_{S}}$ on $S$ coincides with $q_S(\overrightarrow{V_{S_b}})$ for every $b\in B$ and is the average of the boundaries on $q_S(A_e)$ for every $e\in E$.

{\bf II.} Let us prove that the Euler characteristic of $S$ can be calculated by the formula (\ref{EjHar})
$\chi(S)=\sum\limits_{\mathcal M}(X_{\mathcal M}-Y_{\mathcal M})+N_\mathcal A$, 
where $X_{\mathcal M}$ is the result of applying the formula (\ref{EjHarBezPer}) to the four-colour graph $\Gamma_{\mathcal M}$ corresponding to the vertex $\mathcal M$, $Y_{\mathcal M}$ is the quantity of the edges which are incident to $\mathcal M$, $N_\mathcal A$ is the quantity of $\mathcal A$-vertex of $\Upsilon^*$.

It is well-known (see, for example, \cite{BorBlizIzrFom}) that $\chi(\Pi_p)=\chi(\Pi)-p$, where $\Pi_p$ is the surface $\Pi$ with $p$ holes and if $\Pi$ is a result of an identifying of the boundaries of $\Pi^1_p$ and $\Pi^2_p$ then $\chi(\Pi)=\chi(\Pi_p^1)+\chi(\Pi_p^2)$. As $S$ is a result of the identifying of the boundaries of $\bigsqcup\limits_{b\in B} S_{b}$ and $\bigsqcup\limits_{e\in E}A_e$ and $\chi(A_e)=0$ then to calculate $\chi(S)$ we need to calculate the characteristic of its elementary regions and to summarize them. 
As $\chi(S_b)=1$ for $b$ being $\mathcal A$-vertex, $\chi(S_b)=0$ for $b$ being $\mathcal E$- or $\mathcal L$-vertex and $\chi(S_b)=X_{\mathcal M}-Y_{\mathcal M}$ for $b$ being $\mathcal M$-vertex then we get the result. 

{\bf III.} Let us prove that $S$ is orientable if and only if every four-colour graph equipping $\Upsilon^*$ has not odd length cycles and each $\mathcal L$-vertex is incident to exactly two edges.

Notice that $\mathcal S$ is orientable if and only if all its parts are orientable, i.e. all its elementary regions are orientable, that equivalently the condition that all $\mathcal L$-regions are the annuli and all four colour graphs equipping $\Upsilon^*$ do not have odd length cycles (see item $(2)$ of Lemma \ref{ReBezPer}). 

\begin{flushright}$\square$\end{flushright}

\section{Efficient algorithms to solve the isomorphism problem in the classes of four-colour and equipped graphs, to calculate the Euler characteristic and to determine orientability of the ambient surface}

In this section, we consider the distinction (isomorphism) problem for four-colour and equipped graphs and the problems of calculation of the Euler characteristic of
the ambient surface and determining its orientability. We present polynomial-time algorithms for their solution.

\subsection{The isomorphism problem, a proof of Theorem \ref{isomorphism for equipped graphs}}

For two given four-colour (or equipped) graphs, the problem is to
decide whether these graphs are isomorphic or not.
Recall that four-colour graphs and directed graphs of flows can be embedded into the ambient surface.
In other words, these graphs can be depicted on the ambient surface such that their vertices are
points and their edges are Jordan curves on the surface, and no two edges are crossing in an internal point.
This observation is useful for our purposes, as there exists a polynomial-time algorithm for
the isomorphism problem of simple graphs embeddable into a fixed surface.

\begin{definition}
An unlabeled graph without loops, directed and multiple edges is called \emph{simple}.
\end{definition}

\begin{proposition}\cite{M80} \label{Miller}
The isomorphism problem for $n$-vertex simple graphs each embeddable into a surface of genus $g$ can be solved
in $O(n^{O(g)})$ time.
\end{proposition}

First, let us consider only the case of four-colour graphs. We cannot directly apply Proposition \ref{Miller} for
distinction of four-colour graphs, as they are not simple. Nevertheless, it is possible to reduce the problem
for four-colour graphs to the same problem for simple graphs with a small complexity of the reduction. To this end,
we need the following operations with graphs.

\begin{definition}
The operation of \emph{$k$-subdivision of an edge} $(a,b)$ is to delete the edge
from a graph, add vertices $c_1,\ldots,c_k$ and edges $(a,c_1),(c_1,c_2),\ldots,(c_k,b)$.
\end{definition}

\begin{definition}
The operation of \emph{$(k_1,k_2)$-subdivision of an edge} $(a,b)$ is to delete it from a graph, add vertices
$c_1,c_2,\ldots,c_{k_1},v,u,w,d_1,d_2,\ldots,d_{k_2}$ and edges $(a,c_1),\\(c_1,c_2),\ldots,(c_{k_1},v),(v,u),(u,w),(v,w),(v,d_1),(d_1,d_2),\ldots,(d_{k_2},b)$.
\end{definition}

For the four-colour graph $\Gamma_{{\mathcal M}}$ of a given flow, we construct a simple graph $\Gamma^*_{{\mathcal M}}$ as follows.
In the graph $\Gamma_{{\mathcal M}}$ we perform 1-subdivision of each $s$-edge, 2-subdivision of each $t$-edge, 3-subdivision
of each $u$-edge. Let $e=(a,b)$ be an arbitrary $c$-edge of $\Gamma_{\mathcal{M}}$, $num_a(e)$ and $num_b(e)$ be the numbers
of $e$ in the sets of $c$-edges incident to $a$ and $b$, correspondingly. We perform $(num_a(e)+5,num_b(e)+5)$-subdivision of $e$.
A similar operation is performed for all $c$-edges of the graph $\Gamma_{{\mathcal M}}$. The resultant graph $\Gamma^*_{{\mathcal M}}$ is simple (see Fig \ref{AlgRazb}).

\begin{figure}[htb]
\centerline{\includegraphics [width=15 cm] {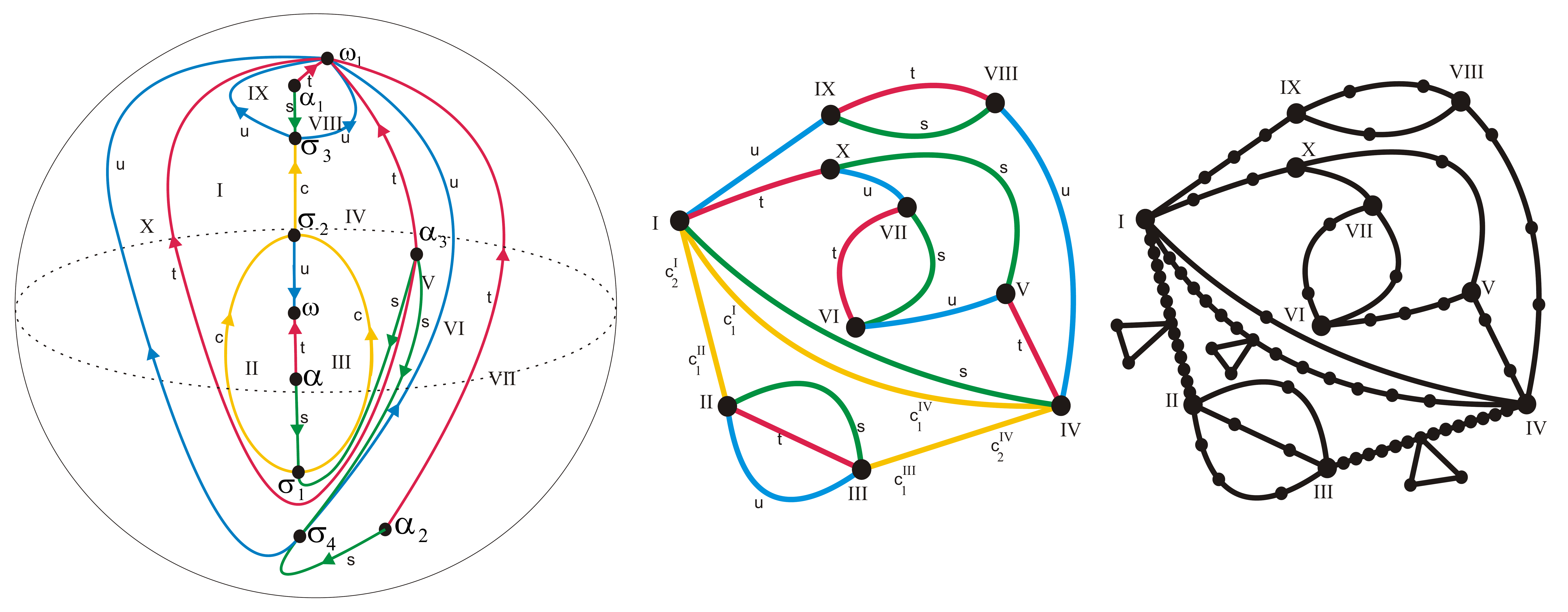}} \caption{$f^t$, $\Gamma_{\mathcal M}$ and $\Gamma^*_{{\mathcal M}}$}\label{AlgRazb}
\end{figure}

\begin{lemma} \label{4-colour graphs}
Graphs $\Gamma_{\mathcal M}$ and $\Gamma_{\mathcal M'}$ are isomorphic iff graphs $\Gamma^*_{\mathcal M}$ and $\Gamma^*_{\mathcal M'}$ are isomorphic.
\end{lemma}

\begin{proof}
Obviously, the graphs $\Gamma^*_{\mathcal M}$ and $\Gamma^*_{\mathcal M'}$ can be uniquely constructed with the graphs $\Gamma_{\mathcal M}$ and $\Gamma_{\mathcal M'}$. Let us show that the opposite fact is also true. It will follow the lemma.

Each polygonal region of $\Delta_{f^t}$ has at least three sides and, therefore, every vertex of $\Gamma_{\mathcal M}$ has at least three neighbours in this graph. Clearly, in the graph $\Gamma^*_{\mathcal M}$ none of the vertices of the graph $\Gamma_{\mathcal M}$ belongs to a triangle.
Hence, the set of vertices of $\Gamma_{\mathcal M}$ consists of those and only those vertices of $\Gamma^*_{\mathcal M}$ that
have at least three neighbours and do not belong to triangles. Deleted all vertices of $\Gamma_{\mathcal M}$ from $\Gamma^*_{\mathcal M}$, we obtain the disjunctive union of connected subgraphs, each of which is a path or a path
with a triangle joined to an internal vertex of the path. These connected subgraphs are indicators of the existence of edges
between the corresponding vertices of $\Gamma_{\mathcal M}$. If a subgraph is a path, then its length determines a
colour in the set $\{s,t,u\}$ of the corresponding edge of $\Gamma_{\mathcal M}$. If a subgraph is a path with
a joined triangle, then it corresponds to some $c$-edge $e=(a,b)$ of $\Gamma_{\mathcal M}$. Deleted all vertices
of the triangle in the subgraph, we obtain two paths, whose lengths show the numbers of $e$ in the sets of $c$-edges
incident to the vertices $a$ and $b$, respectively. Thus, knowing the graph $\Gamma^*_{\mathcal M}$, one can uniquely restore the graph $\Gamma_{\mathcal M}$.
\end{proof}

Let us estimate the number of vertices of $\Gamma^*_{\mathcal M}$, assuming that $\Gamma_{\mathcal M}$
has $n$ vertices and $m$ edges. Clearly, any of $m$ edges of the graph $\Gamma_{\mathcal M}$ corresponds to some subgraph of
the graph $\Gamma^*_{\mathcal M}$ that has at most $2n+18$ vertices. Therefore, the graph $\Gamma^*_{\mathcal M}$
has at most $(2n+18)\cdot m$ vertices and it can computed in polynomial time with the graph $\Gamma_{\mathcal M}$. Notice that
$\Gamma^*_{\mathcal M}$ can be embedded into the ambient surface. By this fact and Lemma \ref{4-colour graphs},
the isomorphism problem for four-colour graphs can be reduced in polynomial time to the same problem for simple graphs,
embedded into a fixed surface. Hence, the following result is true.

\begin{lemma} \label{isomorphism for four-colour graphs}
Isomorphism of four-colour graphs can be recognized in polynomial time.
\end{lemma}

Next, we consider the isomorphism problem for the class of equipped graphs.
Let $\Upsilon^*_{\phi^t}$ be an equipped graph. We will modify it as follows.
We delete all $({\mathcal M},{\mathcal L})$-edges and all $({\mathcal L},{\mathcal M})$-edges
(also forget about their associated $tu$-cycles and $st$-cycles) and replace each ${\mathcal M}$-vertex
by the corresponding graph $\Gamma_{\mathcal M}$. We also connect every
${\mathcal L}$-vertex with all vertices of the associated $tu$-cycle ($st$-cycle)
in the corresponding graph $\Gamma_{\mathcal M}$ by edges oriented as
$({\mathcal M},{\mathcal L})$ (resp. $({\mathcal L},{\mathcal M})$), arrange
orientation of the cycle in $\Gamma_{\mathcal M}$ (preserving colors of its edges)
as it was in $({\mathcal M},{\mathcal L})$ (resp. $({\mathcal L},{\mathcal M})$).
The resultant graph $\Gamma_t$ can be embedded into the ambient surface, as this is true for $\Upsilon^*_{\phi^t}$ and $\Gamma_{\mathcal M}$, for any ${\mathcal M}$-vertex, and by the fact that polygonal regions corresponding to ${\mathcal L}$-vertices and to their neighbours in $tu$-cycles and $st$-cycles have a common border.

We add two degree one neighbours to each ${\mathcal A}$-vertex, three degree one neighbours to each ${\mathcal L}$-vertex,
four degree one neighbours to each ${\mathcal E}$-vertex with the ``-'' weight, and
five degree one neighbours to each ${\mathcal E}$-vertex with the ``+'' weight. Additionally, in
the graph $\Gamma_t$, we perform (2,1)-subdivision of any non-colored oriented edge, (3,1)-subdivision of any
oriented $s$-edge, (4,1)-subdivision of any oriented $t$-edge, (5,1)-subdivision of any
oriented $u$-edge. Finally, for any ${\mathcal M}$, we apply subdivisions of all non-oriented edges
in $\Gamma_{\mathcal M}$ as it was described earlier in the definition of $\Gamma^*_{\mathcal M}$. Clearly, the resultant graph $\Gamma^*_t$ is simple, embeddable
into the ambient surface, and it can be computed in polynomial time.

\begin{lemma} \label{equipped graphs}
Equipped graphs $\Upsilon^*_{\phi^t}$ and $\Upsilon^*_{\phi^{t'}}$ are isomorphic if and only if $\Gamma^*_t$ and $\Gamma^*_{t'}$
are isomorphic.
\end{lemma}

\begin{proof}
Obviously, the graphs $\Gamma^*_t$ and $\Gamma^*_{t'}$ can be uniquely constructed by the graphs $\Upsilon^*_{\phi^t}$ and $\Upsilon^*_{\phi^{t'}}$. Let us show that the opposite fact is also true. It will follow the lemma.

Notice that any vertex of $\Gamma^*_t$ not belonging to ${\mathcal A}\cup {\mathcal L}\cup {\mathcal E}$
has at most one degree one neighbour. Hence, a vertex of $\Gamma^*_t$ is an ${\mathcal A}$-vertex of $\Upsilon^*_{\phi^t}$ iff it has exactly two degree one neighbours; a vertex of $\Gamma^*_t$ is a ${\mathcal L}$-vertex of $\Upsilon^*_{\phi^t}$ iff it has exactly three degree one neighbours; a vertex of $\Gamma^*_t$ is an ${\mathcal E}$-vertex of the weight ``-'' of $\Upsilon^*_{\phi^t}$ iff it has exactly four degree one neighbours;
 a vertex of $\Gamma^*_t$ is an ${\mathcal E}$-vertex of the weight ``+'' of $\Upsilon^*_{\phi^t}$ iff it has exactly five degree one neighbours.  

Therefore, one can determine all ${\mathcal A}$-, ${\mathcal L}$-, ${\mathcal E}$-vertices
of $\Upsilon^*_{\phi^t}$ in the graph $\Gamma^*_t$. Hence, one can determine all $({\mathcal A},{\mathcal L})$-, $({\mathcal L},{\mathcal A})$-, $({\mathcal L},{\mathcal E})$-, and $({\mathcal E},{\mathcal L})$-edges of $\Upsilon^*_{\phi^t}$, knowing their ends and subgraphs of $\Gamma^*_t$ between them. Considering
a ball of radius five centering at a ${\mathcal L}$-vertex, one can determine orientation of the corresponding $({\mathcal L},{\mathcal M})$-edge or
$({\mathcal M},{\mathcal L})$-edge, all vertices of its associated $tu$-cycle or $st$-cycle in the graph $\Upsilon^*_{\phi^t}$.
Deleted all radius four balls centering at vertices in ${\mathcal A}\cup {\mathcal L}\cup {\mathcal E}$, we obtain the disjunctive union
of subgraphs, which are analogues of the graphs of the form $\Gamma^*_{\mathcal M}$. By any such a subgraph, one can determine the corresponding graph $\Gamma_{\mathcal M}$, associated $tu$-cycles and $st$-cycles and their orientation. Thus, knowing the graph $\Gamma^*_t$, it is possible to uniquely restore the graph $\Upsilon^*_{\phi^t}$.
\end{proof}

Recall that the graph $\Gamma^*_t$ is simple, and it can be computed in polynomial time. By this fact and Lemma \ref{equipped graphs},
the isomorphism problem for equipped graphs can be reduced in polynomial time to the same problem for simple graphs
embedded into a fixed surface. Hence, Theorem \ref{isomorphism for equipped graphs} is true.

\subsection{The Euler characteristic and the surface orientability, a proof of Theorem \ref{orientability and Euler characteristic}}

Now, we consider the problems of calculation of the Euler characteristic of
the ambient surface and determining its orientability. To this end, we need the notion
of a bipartite graph.

\begin{definition}
A simple graph is called \emph{bipartite} if the set of its vertices can be partitioned into two parts such that there is no
an edge incident to two vertices in the same part.
\end{definition}

By K\"onig theorem, a simple graph is bipartite if and only if it does not contain odd cycles \cite{K31}.
For any simple graph with $n$ vertices and $m$ edges, its bipartiteness can be recognized
in $O(n+m)$ time by breath-first search \cite{AT}. Hence, by the second part of Theorem \ref{Realiz},
to check orientability of the ambient surface, we forget about colours of edges of four-colour graphs and
apply 2-subdivision to each their edge, to make them simple. Clearly,
all of the new graphs are bipartite if and only if the ambient surface is orientable. Thus,
orientability of the ambient surface can be tested in linear time on the length of a description
of equipped graphs.

By Lemma \ref{ReBezPer}, the Euler characteristic of a surface $M$
is equal to $\nu_0-\nu_1+\nu_{2}$, where $\nu_0, \nu_1, \nu_{2}$ are the numbers of all $tu$-, $c^*$-, and $st$-cycles of the four-colour graph
$\Gamma_{\mathcal M}$ of a flow without limit cycles on $M$, respectively. Deleted all $c$-edges and all
$s$-edges from $\Gamma_{\mathcal M}$, we obtain the disjunctive sum of $tu$-cycles. Similarly, deleting all
$c$-edges and all $u$-edges, we obtain the disjunctive sum of $st$-cycles. Therefore, $\nu_0$ and $\nu_{2}$
can be computed in time proportional to the sum of the numbers of vertices and edges of $\Gamma_{\mathcal M}$.
If an edge $e=(a,b)$ of $\Gamma_{\mathcal M}$ belongs to some its $c^*$-cycle $C$, then the vertex $a$ has an odd or even number in $C$. Hence, assuming that
this number of $a$ is odd (or even) in $C$, by the number of $e$ in the set of edges incident to $b$,
one can determine an edge in $C$ following the edge $e$. Hence, each edge of $\Gamma_{\mathcal M}$ is contained
in at most two $c^*$-cycles and they can be found in time proportional to the number of edges of $\Gamma_{\mathcal M}$.
Found all these cycles, one can remove $e$ from $\Gamma_{\mathcal M}$ and similarly proceed
our search of $c^*$-cycles in the resultant graph. Clearly, the found cycles will not be met one more time in
the future searches of $c^*$-cycles. Therefore, $\nu_1$ can be computed in time proportional to the square of the number of edges of $\Gamma_{\mathcal M}$.
Thus, by the first part of Theorem \ref{Realiz}, the statement of Theorem \ref{orientability and Euler characteristic} holds.

\pagebreak

\noindent Vladislav E. Kruglov \\ HSE Campus in Nizhny Novgorod, Faculty of Informatics, Mathematics, and Computer Science, Laboratory of Topological Methods in Dynamics. Trainee Researcher; \\ Lobachevsky State University of Nizhny Novgorod, Institute ITMM, Department of Mathematical Physics and Optimal Control. Master's student; \\ E-mail: vekruglov@hse.ru, Tel. 433-57-83

~~

~~

~~

\noindent\textbf{Any opinions or claims contained in this Working Paper do not necessarily reflect the views of HSE.}

~~

\noindent\footnotesize{\textbf{© Kruglov, 2017}}

\noindent\footnotesize{\textbf{© Malyshev, 2017}}

\noindent\footnotesize{\textbf{© Pochinka, 2017}}

\end{document}